\documentclass[11pt,reqno]{amsart}

\usepackage{amssymb,amsmath,amsfonts,amsthm}
\usepackage{bm,cite,graphicx,color}
\usepackage[toc,page]{appendix}

\newtheorem{theorem}{Theorem}[section]
\newtheorem{lemma}[theorem]{Lemma}

\newtheorem{assumption}{Assumption}
\numberwithin{equation}{section}

\allowdisplaybreaks[4]

\begin{document}

\title [inverse random potential scattering]{Inverse random potential scattering for elastic waves}

\author{Jianliang Li}
\address{School of Mathematics and Statistics, Changsha University of Science
and Technology, Changsha 410114, P.R. China.}
\email{lijl@amss.ac.cn}

\author{Peijun Li}
\address{Department of Mathematics, Purdue University, West Lafayette, Indiana
47907, USA.}
\email{lipeijun@math.purdue.edu}

\author{Xu Wang}
\address{Department of Mathematics, Purdue University, West Lafayette, Indiana
47907, USA.}
\email{wang4191@purdue.edu}

\thanks{The research of PL is supported in part by the NSF grant DMS-1912704.}



\begin{abstract}
This paper is concerned with the inverse elastic scattering problem for a random potential in three dimensions. Interpreted as a distribution, the potential is assumed to be a microlocally isotropic Gaussian random field whose covariance operator is a classical pseudo-differential operator. Given the potential, the direct scattering problem is shown to be well-posed in the distribution sense by studying the equivalent Lippmann--Schwinger integral equation. For the inverse scattering problem, we demonstrate that the microlocal strength of the random potential can be uniquely determined with probability one by a single realization of the high frequency limit of the averaged compressional or shear backscattered far-field pattern of the scattered wave. The analysis employs the integral operator theory, the Born approximation in the high frequency regime, the microlocal analysis for the Fourier integral operators, and the ergodicity of the wave field.
\end{abstract}

\maketitle

\section{Introduction}

Inverse problems are to seek the causal factors which produce observations. Mathematically, they are to determine unknown parameters in partial differential equations via external measurements. The field has undergone a tremendous growth in the last several decades. Motivated by diverse scientific and industrial applications such as radar and sonar, nondestructive testing, and medical imaging, the inverse scattering has progressed to an area of intense activity and is in the foreground of mathematical research in scattering theory \cite{CK19}. The inverse scattering problems for elastic waves have received ever-increasing attention due to the significant applications in geophysics, seismology, and elastography \cite{ABG-15}. We refer to \cite{BHSY,BLZ,BC} and the references cited therein for mathematical and computational results on inverse problems in elasticity. Stochastic inverse problems refer to inverse problems that involve uncertainties, which become essential for mathematical models in order to take account of unpredictability of the environments, incomplete knowledge of the systems and measurements, and interference between different scales. In addition to the existing hurdles of nonlinearity and ill-posedness for deterministic counterparts, stochastic inverse problems have substantially more difficulties due to randomness and uncertainties. The area is widely open, and new models and methodologies are highly desired in related applications of the stochastic inverse problems. Recent progress can be found in \cite{BCL18,LHL,LL,LW1,LW2} on the the inverse random source problems for the wave equations, where the goal is to determine the statistical properties of the random source from a knowledge the radiated random wave field. 

Due to the nonlinearity, the inverse random potential problem is more challenging than the inverse random source problem. There are only few mathematical results. In \cite{LPS08}, the authors consider the point source illumination and study the inverse random potential problem for the two-dimensional 
Schr\"{o}dinger equation, where the potential is assumed to be a microlocally isotropic Gaussian random field (cf. Assumption \ref{as:rho}). It is shown that the microlocal strength of the random potential can be uniquely determined by high frequency limit of the scattered wave averaged over the frequency band. The work has been extended in \cite{LLW} to the more complicated stochastic elastic wave equation in $\mathbb R^n$ with $n=2, 3$. The microlocal strength of the random potential is shown to be uniquely determined by the high frequency limit of $n$ scattered fields averaged over the frequency band, where the $n$ scattered fields are excited by $n$ point sources with $n$ unit orthonormal polarization vectors. When the incidence is taken to be the plane wave, the inverse random potential problem is 
examined for the three-dimensional Schr\"{o}dinger equation in \cite{CHL}. The authors demonstrate that the microlocal strength of the random potential can be uniquely recovered by the high frequency limit of the backscattered far-field pattern averaged over the frequency band. In \cite{LLM-SIMA,LLM}, the authors consider a more complex model and study the inverse random potential problem for the three-dimensional Schr\"{o}dinger equation where both the potential and source are random. A related inverse acoustic scattering problem for a random impedance can be found in \cite{HLP}. 

This paper is concerned with an inverse scattering problem for the three-dimensional elastic wave equation with a random potential. Specifically, consider the stochastic elastic wave equation 
\begin{eqnarray}\label{a1}
\mu\Delta \boldsymbol{u}+(\lambda+\mu)\nabla\nabla\cdot
\boldsymbol{u}+\omega^2\boldsymbol{u}-\rho\boldsymbol{u}=0 \quad {\rm in} ~ \mathbb R^{3},
\end{eqnarray}
where $\omega>0$ is the angular frequency, $\lambda$ and $\mu$ are the Lam\'{e} constants satisfying $\mu>0$ and $\lambda+2\mu>0$ such that the second order partial differential operator $\Delta^*:=\mu\Delta+(\lambda+\mu)\nabla\nabla\cdot$ is strongly elliptic (cf. \cite[Section 10.4]{M}), and the random potential $\rho$ stands for a linear load inside a known homogeneous and isotropic elastic solid, and is assumed to be a generalized Gaussian random field whose definition is given as follows. 
 
 \begin{assumption}\label{as:rho}
 Assume that the random potential $\rho$ is a  centered microlocally isotropic Gaussian random field of order $-m$ in $D$ with $m\in (2, 3]$ and $D\subset\mathbb R^3$ being a bounded domain. The covariance operator ${\mathcal C}_\rho$ of $\rho$, defined by
\[
\langle\mathcal C_\rho\varphi,\psi\rangle:=\mathbb E[\langle\rho,\varphi\rangle\langle\rho,\psi\rangle]\quad\forall~\varphi,\psi\in\mathcal{D},
\]
is a classical pseudo-differential operator, and its principal symbol has the form $\phi(x)|\xi|^{-m}$, where $\phi$ is called the microlocal strength of $\rho$ and satisfies $\phi\in C_0^{\infty}(D)$ and $\phi\geq 0$. 
 \end{assumption}
 
The total field $\boldsymbol{u}$ consists of the scattered field $\boldsymbol{u}^{sc}$ and the incident field $\boldsymbol{u}^{inc}$, which is assumed to be the elastic plane wave of the general form 
\begin{eqnarray}\label{a2}
\boldsymbol{u}^{inc}(x)=\alpha\boldsymbol u^{inc}_{\rm p}(x)+\beta\boldsymbol u^{inc}_{\rm s}(x),\quad \alpha,\beta\in{\mathbb C}.
\end{eqnarray}
Here, $\boldsymbol{u}^{inc}_{\rm p}:=\theta e^{{\rm i}\kappa_p x\cdot \theta}$ is the compressional plane wave and $\boldsymbol{u}^{inc}_{\rm s}:=\theta^{\perp}e^{{\rm i}\kappa_s x\cdot \theta}$ denotes the shear plane wave, where $\theta\in{\mathbb S}^2:=\{x\in{\mathbb R}^3: |x|=1\}$ represents the unit propagation direction, $\theta^{\perp}\in\mathbb S^2$ is a vector orthogonal to $\theta$, and $\kappa_{\rm p}:=c_{\rm p}\omega$ and $\kappa_{\rm s}:=c_{\rm s}\omega$ with $c_{\rm p}:=(\lambda+2\mu)^{-\frac{1}{2}}$ and $c_{\rm s}:=\mu^{-\frac{1}{2}}$ denote the compressional and shear wave numbers, respectively. In this paper, we consider separately these two types of incident plane waves: one is the compressional plane wave $\boldsymbol{u}^{inc}=\boldsymbol{u}_{\rm p}^{inc}$ with $\alpha=1$ and $\beta=0$; another is the shear plane wave $\boldsymbol{u}^{inc}=\boldsymbol{u}_{\rm s}^{inc}$ with $\alpha=0$ and $\beta=1$. It can be verified that the incident field $\boldsymbol u^{inc}$ satisfies
\[
 \mu\Delta \boldsymbol{u}^{inc}+(\lambda+\mu)\nabla\nabla\cdot \boldsymbol{u}^{inc}+\omega^2\boldsymbol{u}^{inc}=0 \quad {\rm in} ~ \mathbb R^{3}. 
\]

Since the problem is formulated in the whole space, an appropriate radiation condition is needed to ensure the uniqueness of the solution. As usual, the scattered field $\boldsymbol{u}^{sc}$ is required to satisfy the Kupradze--Sommerfeld radiation condition. Based on the Helmholtz decomposition (cf. \cite[Appendix B]{BLZ}), the scattered field $\boldsymbol{u}^{sc}$ can be decomposed into the compressional wave component $\boldsymbol{u}^{sc}_{\rm p}:=-\frac{1}{\kappa_{\rm p}^2}\nabla\nabla\cdot\boldsymbol{u}^{sc}$ and the shear wave component $\boldsymbol{u}^{sc}_{\rm s}:=\frac{1}{\kappa_{\rm s}^2}{\rm{\bf curl}}{\rm curl}\boldsymbol{u}^{sc}$ in $\mathbb R^3\setminus\overline D$. The Kupradze--Sommerfeld radiation condition reads that $\boldsymbol{u}_{\rm p}^{sc}$ and $\boldsymbol{u}_{\rm s}^{sc}$ satisfy the Sommerfeld radiation condition 
\begin{eqnarray}\label{a3}
\lim_{|x|\rightarrow\infty}|x|\left(\partial_{|x|}
\boldsymbol{u}^{sc}_{\rm p} -{\rm i}\kappa_{\rm p}\boldsymbol{u}^{sc}_{\rm
p}\right)=0,\quad \lim_{|x|\rightarrow\infty}|x|\left(\partial_{|x|}
\boldsymbol{u}^{sc}_{\rm s} -{\rm i}\kappa_{\rm s}\boldsymbol{u}^{sc}_{\rm
s}\right)=0
\end{eqnarray}
uniformly in all directions $\hat{x}:=x/|x|\in {\mathbb S^2}$. The radiation condition (\ref{a3}) leads to the following asymptotic expansion of $\boldsymbol{u}^{sc}$:
\begin{eqnarray}\label{a4}
\boldsymbol{u}^{sc}(x)=\frac{e^{{\rm i}\kappa_{\rm p}|x|}}{|x|}\boldsymbol{u}_{\rm p}^{\infty}(\hat{x})+\frac{e^{{\rm i}\kappa_{\rm s}|x|}}{|x|}\boldsymbol{u}_{\rm s}^{\infty}(\hat{x})+O(|x|^{-2}),\quad |x|\to\infty,
\end{eqnarray}
where $\boldsymbol{u}_{\rm p}^{\infty}(\hat{x})$ and $\boldsymbol{u}_{\rm s}^{\infty}(\hat{x})$ are known as the compressional and shear far-field patterns of the scattered field $\boldsymbol{u}^{sc}$, respectively.

Note that the wave fields $\boldsymbol u,\boldsymbol u^{\infty}_{\rm p},\boldsymbol u^{\infty}_{\rm s}$ also depend on the angular frequency $\omega$ and the propagation direction $\theta$. For clarity, we write $\boldsymbol v(x)$ as $\boldsymbol v(x,\omega,\theta)$ when it is necessary to express explicitly the dependence of the wave field $\boldsymbol v$ on $\omega$ and $\theta$.

Given the random potential $\rho$, the direct scattering problem is to investigate the well-posedness and regularity of the solution $\boldsymbol u$ to (\ref{a1})--(\ref{a3}). The inverse scattering problem aims to determine the microlocal strength $\phi$ of the random potential from a knowledge of the wave field $\boldsymbol u$. In this work, we consider both the direct and inverse scattering problems. The direct scattering problem is shown to be well-posed in the distribution sense (cf. Theorem \ref{thm2}). Below, we present the main result on the uniqueness of the inverse scattering problem and outline the steps of its proof for readability.

\begin{theorem}\label{thm1}
Let the random potential $\rho$ satisfy the Assumption \ref{as:rho} with $m\in(\frac{14}5,3]$. Denote by $\boldsymbol{u}_{\rm p}^{\infty}$ and $\boldsymbol{u}_{\rm s}^{\infty}$ the compressional and shear far-field patterns of the scattered wave $\boldsymbol u^{sc}$ associated with $\boldsymbol{u}^{inc}=\boldsymbol{u}^{inc}_{\rm p}$ and $\boldsymbol{u}^{inc}=\boldsymbol{u}^{inc}_{\rm s}$, respectively. For all $\theta\in {\mathbb S}^2$ and $\tau\geq 0$, it holds almost surely that 
\begin{align}\label{a5}
\lim_{Q\to\infty}\frac{1}{Q}\int_{Q}^{2Q}\omega^{m}\boldsymbol{u}_{\rm p}^{\infty}(-\theta,\omega,\theta)\cdot\overline{\boldsymbol{u}_{\rm p}^{\infty}(-\theta,\omega+\tau,\theta)}d\omega &=C_{\rm p}\hat{\phi}(2c_{\rm p}\tau \theta),\\\label{a6}
\lim_{Q\to\infty}\frac{1}{Q}\int_{Q}^{2Q}\omega^{m}\boldsymbol{u}_{\rm s}^{\infty}(-\theta,\omega,\theta)\cdot\overline{\boldsymbol{u}_{\rm s}^{\infty}(-\theta,\omega+\tau,\theta)}d\omega &=C_{\rm s}\hat{\phi}(2c_{\rm s}\tau \theta),
\end{align}
where $C_{\rm p}=2^{-m-4}\pi^{-2}c_{\rm p}^{4-m}$, $C_{\rm s}=2^{-m-4}\pi^{-2}c_{\rm s}^{4-m}$, and $
\hat{\phi}(\xi)=\mathcal{F}[\phi](\xi)=\int_{\mathbb R^3}\phi(x)e^{-{\rm i}x\cdot\xi}dx$
is the Fourier transform of $\phi$. Moreover, the microlocal strength $\phi$ is uniquely determined from (\ref{a5}) or (\ref{a6}) with $(\tau,\theta)\in\Theta$ and $\Theta\subset\mathbb R_+\times\mathbb S^2$ being any open domain.
\end{theorem}

Since the potential is random, the scattered wave and its far-field pattern are also random fields. In general, the scattering data used to recover the random coefficients involved in the stochastic inverse problems depend on the realizations of the random coefficients. Interestingly, the results in Theorem \ref{thm1} demonstrate that the scattering data given on the left-hand side of (\ref{a5})--(\ref{a6}) are statistically stable, i.e., they are independent of the realizations of the potential. The compressional or shear backscattered far-field pattern generated by any single realization of the random potential can determine with probability one the microlocal strength $\phi$ of the random potential.

To prove Theorem \ref{thm1}, we consider the equivalent Lippmann--Schwinger integral equation and show that the solution can be written as a Born series for sufficiently large frequency, i.e., 
\begin{eqnarray*}
\boldsymbol{u}^{sc}(x,\omega,\theta)=\sum_{j=1}^{\infty}\boldsymbol{u}_j(x,\omega, \theta).
\end{eqnarray*} 
Correspondingly, the far-field pattern $\boldsymbol{u}^{\infty}:=\boldsymbol{u}_{\rm p}^{\infty}+\boldsymbol{u}_{\rm s}^{\infty}$ of the scattered field $\boldsymbol{u}^{sc}$ has the form
\begin{align*}
\boldsymbol{u}^{\infty}(\hat{x},\omega,\theta)=\boldsymbol{u}_1^{\infty}(\hat{x},\omega, \theta)+\boldsymbol{u}_2^{\infty}(\hat{x},\omega, \theta)+\boldsymbol{b}(\hat{x},\omega, \theta),
\end{align*} 
where $\boldsymbol{b}(\hat{x},\omega,\theta):=\sum_{j=3}^{\infty}\boldsymbol{u}_j^{\infty}(\hat{x},\omega, \theta)$ and $\boldsymbol{u}_j^{\infty}$ denotes the far-field pattern of $\boldsymbol{u}_j$. For the first order far-field pattern $\boldsymbol{u}_1^{\infty}$, we show by using the Fourier analysis in Section \ref{sec:3.1} that 
\begin{align*}
\lim_{Q\to\infty}\frac{1}{Q}\int_{Q}^{2Q}\omega^{m}\boldsymbol{u}_{1,\rm p}^{\infty}(-\theta,\omega,\theta)\cdot\overline{\boldsymbol{u}_{1,\rm p}^{\infty}(-\theta,\omega+\tau,\theta)}d\omega &=C_{\rm p}\hat{\phi}(2c_{\rm p}\tau \theta),\\\label{a10}
\lim_{Q\to\infty}\frac{1}{Q}\int_{Q}^{2Q}\omega^{m}\boldsymbol{u}_{1,\rm s}^{\infty}(-\theta,\omega,\theta)\cdot\overline{\boldsymbol{u}_{1,\rm s}^{\infty}(-\theta,\omega+\tau,\theta)}d\omega &=C_{\rm s}\hat{\phi}(2c_{\rm s}\tau \theta),
\end{align*}
where $\boldsymbol{u}_{j,\rm p}^{\infty}$ and $\boldsymbol{u}_{j,\rm s}^{\infty}$ are the compressional and shear far-field patterns of $\boldsymbol{u}_j^{\infty}$ for $j\in{\mathbb N}$. For the second order far-field pattern $\boldsymbol{u}_2^{\infty}$, the higher order far-field pattern $\boldsymbol b$, and their interactions to the first order far-field pattern, we employ microlocal analysis of Fourier integral operators and show that they are negligible in Sections \ref{sec:3.2} and \ref{sec:3.3}, i.e.,
\begin{align*}
\lim_{Q\to\infty}\frac{1}{Q}\int_{Q}^{2Q}\omega^{m}|\boldsymbol{u}_2^{\infty}(-\theta,\omega, \theta)|^2d\omega = 0,\quad
\lim_{Q\to\infty}\frac{1}{Q}\int_{Q}^{2Q}\omega^{m}|\boldsymbol{b}(-\theta,\omega, \theta)|^2d\omega = 0.
\end{align*}

The paper is organized as follows. In Section 2, the well-posedness is established for the direct scattering problem by studying the equivalent Lippmann--Schwinger integral equation; the convergence of the series solution is proved for the 
Lippmann--Schwinger integral equation for sufficiently large frequency. Section 3 is devoted to the inverse scattering problem, where a uniqueness result is obtained to determine the microlocal strength of the random potential. The paper concludes with some general remarks in Section 4.

\section{The direct scattering problem}

In this section, we address the well-posedness of the scattering problem (\ref{a1})--(\ref{a3}) and the regularity of the solution $\boldsymbol u$. The challenge arises from the roughness of the random potential $\rho$. By the following lemma, the potential $\rho$ should be interpreted as a distribution in $W^{\frac{m-3}{2}-\epsilon, p}(D)$ almost surely for any $\epsilon>0$ and $p\in (1, \infty)$. The proof of Lemma \ref{lemmaadd3} can be found in \cite{LW1,LPS08}. 

\begin{lemma}\label{lemmaadd3}
Let $\rho$ be a microlocally isotropic Gaussian random field of order $-m$ in $D\subset\mathbb R^n$ with $m\in[0,n+2)$. 
\begin{itemize}
\item[(i)] If $m\in(n,n+2)$, then $\rho\in C^{0,\eta}(D)$ almost surely for all $\eta\in(0,\frac{m-n}2)$.

\item[(ii)] If $m\in[0,n]$, then $\rho\in W^{\frac{m-n}2-\epsilon,p}(D)$ almost surely for any $\epsilon>0$ and $p\in(1,\infty)$.
\end{itemize}
\end{lemma}

Since the potential $\rho$ is a distribution, the well-posedness of the problem (\ref{a1})--(\ref{a3}) is examined in the distribution sense by studying the equivalent Lippmann--Schwinger integral equation 
\begin{equation}\label{b1}
(\mathcal I+{\mathcal K}_\omega)\boldsymbol u=\boldsymbol u^{inc},
\end{equation}
where $\mathcal I$ is the identity operator and the operator ${\mathcal K}_{\omega}$ is defined by 
\begin{eqnarray*}
({\mathcal K}_{\omega}\boldsymbol{u})(x):=\int_{\mathbb R^3}\boldsymbol{G}(x,z,\omega)\rho(z)\boldsymbol{u}(z)dz.
\end{eqnarray*}
Here, $\boldsymbol{G}\in {\mathbb C}^{3\times 3}$ denotes the Green tensor for the elastic wave equation and is given by
\begin{eqnarray}\label{b2}
\boldsymbol{G}(x,z,\omega)=\frac{1}{\mu}\Phi(x,z,\kappa_{\rm
s})\boldsymbol{I}+\frac{1}{\omega^2}\nabla_x\nabla_x^\top\Big[\Phi(x,z,
\kappa_{\rm s})-\Phi(x,z,\kappa_{\rm p})\Big],
\end{eqnarray}
where $\boldsymbol{I}$ is the $3\times 3$ identity matrix and $\Phi=\frac{e^{{\rm i}\kappa|x-z|}}{4\pi |x-z|}$ is the fundamental solution of the three-dimensional Helmholtz equation.

In the sequel, we denote by $\boldsymbol V=\{\boldsymbol v=(v_1,v_2,v_3)^\top: v_i\in V,~i=1,2,3\}$ the Cartesian product vector space of the space $V$. The notation $a\lesssim b$ stands for $a\leq Cb$, where $C$ is a positive constant whose value is not required and may change step by step in the proofs.  

\begin{theorem}\label{thm2}
Let $\rho$ satisfy Assumption \ref{as:rho}. Then the scattering problem (\ref{a1})--(\ref{a3}) is well-defined in the distribution sense, and admits a unique solution $\boldsymbol{u}\in \boldsymbol W_{\rm loc}^{\gamma, q}({\mathbb R^3})$ almost surely with $q\in(2,\frac6{7-2m})$ and $\gamma\in (\frac{3-m}2, \frac3{2q}-\frac14)$.
\end{theorem}

\begin{proof}
To address the existence of the solution to the scattering problem (\ref{a1})--(\ref{a3}), we first show that the Lippmann--Schwinger equation (\ref{b1}) admits a unique solution in $\boldsymbol W_{\rm loc}^{\gamma, q}({\mathbb R^3})$, and then prove that the solution to (\ref{b1}) is also a solution to (\ref{a1})--(\ref{a3}) in the distribution sense.

By \cite[Lemma 3.1]{LW3}, the operator ${\mathcal K}_{\omega}:\boldsymbol W^{\gamma,q}(U)\to\boldsymbol W^{\gamma,q}(U)$ is compact, where $\gamma\in(\frac{3-m}2,\frac3q-\frac12)$ and $U\subset\mathbb R^3$ is any bounded open set with a locally Lipschitz boundary. Noting that the incident wave $\boldsymbol{u}^{inc}$ given in (\ref{a2}) is smooth in $\mathbb R^3$, we have $\boldsymbol{u}^{inc}\in\boldsymbol W^{\gamma,q}(U)$. It follows from the Fredholm alternative theorem that (\ref{b1}) admits a unique solution $\boldsymbol u\in\boldsymbol W_{\rm loc}^{\gamma,q}(\mathbb R^3)$ (cf. \cite{LL,LW3}).

Next is to show that the solution $\boldsymbol u$ obtained above is also a solution to the scattering problem (\ref{a1})--(\ref{a3}). Denote by $\boldsymbol{\mathcal{D}}$ the space $\boldsymbol C_0^{\infty}(\mathbb R^3)$ equipped with a locally convex topology, which is also known as the space of test functions, and by $\langle\cdot,\cdot\rangle$ the following dual product between a pair of dual spaces $\boldsymbol V$ and $\boldsymbol V^*$:
\[
\langle\boldsymbol v,\boldsymbol w\rangle:=\int_{\mathbb R^3}\boldsymbol v(x)^\top\overline{\boldsymbol w(x)}dx\quad \forall~\boldsymbol v\in\boldsymbol V,~\boldsymbol w\in\boldsymbol V^*.
\]
Noting that $\boldsymbol u$ satisfies $\boldsymbol u=\boldsymbol u^{inc}-{\mathcal K}_{\omega}\boldsymbol u$, we have for any $\boldsymbol{\psi}\in\boldsymbol{\mathcal{D}}$ that 
\begin{eqnarray*}
&&\langle \Delta^*\boldsymbol{u}+\omega^2\boldsymbol{u}-\rho\boldsymbol{u}, \boldsymbol{\psi}\rangle\\
&=&\langle \Delta^*\boldsymbol{u}^{inc}+\omega^2\boldsymbol{u}^{inc}, \boldsymbol{\psi}\rangle
-\big\langle \int_{\mathbb R^3}(\Delta^*+\omega^2)\boldsymbol{G}(\cdot,z,\omega)\rho(z)\boldsymbol{u}(z)dz, \boldsymbol{\psi}\big\rangle-\langle \rho\boldsymbol{u}, \boldsymbol{\psi}\rangle\\
&=&-\int_{\mathbb R^3}\rho(z)\boldsymbol{u}(z)^{\top}\langle (\Delta^*+\omega^2)\boldsymbol{G}(\cdot,z,\omega),\boldsymbol{\psi}\rangle dz-\langle \rho\boldsymbol{u}, \boldsymbol{\psi}\rangle\\
&=&\int_{\mathbb R^3}\rho(z)\boldsymbol{u}(z)^{\top}\overline{\boldsymbol{\psi}(z)}dz-\langle \rho\boldsymbol{u}, \boldsymbol{\psi}\rangle=0,
\end{eqnarray*}
where we use the facts that $\Delta^*\boldsymbol{u}^{inc}+\omega^2\boldsymbol{u}^{inc}=0$ and $ (\Delta^*+\omega^2)\boldsymbol{G}(x,z,\omega)=-\delta(x-z)\boldsymbol I$ with $\delta$ being the Dirac delta function. Thus, $\boldsymbol{u}$ satisfies the equation (\ref{a1}). Moreover, (\ref{b1}) implies that the scattered wave $\boldsymbol u^{sc}$ has the form
\begin{eqnarray*}\label{b6}
\boldsymbol{u}^{sc}(x)=-\int_{{\mathbb R^3}}\boldsymbol{G}(x,z,\omega)\rho(z)\boldsymbol{u}(z)dz, 
\end{eqnarray*}
which satisfies the Kupradze--Sommerfeld radiation condition (\ref{a3}) since $\boldsymbol{G}(\cdot, z,\omega)$ satisfies the Kupradze--Sommerfeld radiation condition (\ref{a3}). Hence, $\boldsymbol u$ also a solution of the scattering problem (\ref{a1})--(\ref{a3}).

The uniqueness follows directly from the proof of \cite[Theorem 4.3]{LW3}, which requires in addition $\gamma<\frac3{2q}-\frac14$ and concludes that the scattering problem (\ref{a1})--(\ref{a3}) is equivalent to the Lippmann--Schwinger integral equation.
\end{proof}

Due to the equivalence between the scattering problem (\ref{a1})--(\ref{a3}) and the Lippmann--Schwinger integral equation (\ref{b1}), we only need to consider the Lippmann--Schwinger integral equation (\ref{b1}) in order to study the regularity of the solution. 

Define the Born sequence 
\begin{eqnarray}\label{b9}
\boldsymbol{u}_j(x)=-({\mathcal K}_{\omega}\boldsymbol{u}_{j-1})(x),\quad j\in{\mathbb N},
\end{eqnarray}
where the leading term 
\begin{eqnarray*}\label{b10}
\boldsymbol{u}_0(x)=\boldsymbol{u}^{inc}(x).
\end{eqnarray*}
The rest of this section is to show that, for sufficiently large frequency $\omega$, the Born series $\sum_{j=0}^{\infty}\boldsymbol{u}_j$ converges to the solution $\boldsymbol{u}$ of the scattering problem (\ref{a1})--(\ref{a3}). 

Introduce the following weighted $L^p$ space (cf. \cite{LLM,LLW}): 
\begin{eqnarray*}
L^p_{\zeta}(\mathbb R^3):=\left\{f\in L_{\rm loc}^1(\mathbb R^3): \|f\|_{L^p_{\zeta}(\mathbb R^3)}<\infty\right\},
\end{eqnarray*}
where 
\begin{eqnarray*}\label{b11}
\|f\|_{L^p_{\zeta}(\mathbb R^3)}:=\|(1+|\cdot|^2)^{\frac{\zeta}{2}}f\|_{L^p(\mathbb R^3)}=\left(\int_{\mathbb R^3}(1+|x|^2)^{\frac{\zeta p}{2}}|f(x)|^pdx\right)^{\frac{1}{p}}.
\end{eqnarray*}
Let $\mathcal{S}$ be the set of all rapidly decreasing functions on $\mathbb R^3$ and $\mathcal{S'}$ denote the dual space of $\mathcal{S}$. Define the space 
\begin{eqnarray*}\label{b12}
H^{s,p}_{\zeta}(\mathbb R^3):=\left\{f\in\mathcal{S'}: (\mathcal I-\Delta)^{\frac{s}{2}}f\in L^p_{\zeta}(\mathbb R^3)\right\},
\end{eqnarray*}
which is equipped with the norm
\begin{eqnarray*}\label{b13}
\|f\|_{H^{s,p}_{\zeta}(\mathbb R^3)}=\|(\mathcal I-\Delta)^{\frac{s}{2}}f\|_{ L^p_{\zeta}(\mathbb R^3)}.
\end{eqnarray*}

We use the notation $H^{s}_{\zeta}(\mathbb R^3):=H^{s,2}_{\zeta}(\mathbb R^3)$ if in particular $p=2$. Moreover, the space $H_0^{s,p}(\mathbb R^3)$ coincides with the classical Sobolev space $W^{s,p}(\mathbb R^3)$.  These definitions enable us to present the following result which gives the estimates for the operator ${\mathcal K}_{\omega}$.

\begin{lemma}\label{lemma1}
Let $s\in (0, \frac{1}{2})$ and the potential $\rho$ satisfy Assumption \ref{as:rho} with $m\in (3-2s, 3]$. Then the following estimates hold: 
\begin{align*}
\|{\mathcal K}_{\omega}\|_{\mathcal{L}(\boldsymbol H^{s}_{-1}(\mathbb R^3),\boldsymbol H^{s}_{-1}(\mathbb R^3))} &\lesssim \omega^{-1+2s},\\
\|{\mathcal K}_{\omega}\|_{\mathcal{L}(\boldsymbol H^{s}_{-1}(\mathbb R^3),\boldsymbol L^{\infty}(\mathcal V))} &\lesssim \omega^{s+\epsilon+\frac{1}{2}},
\end{align*}
where $\mathcal V\subset\mathbb R^3$ is a bounded domain and $\epsilon>0$ is an arbitrary constant. 
\end{lemma}

\begin{proof}
Define the operator 
\[
({\mathcal H}_{\omega}\boldsymbol v)(x):=\int_{\mathbb R^3}\boldsymbol G(x,z,\omega)\boldsymbol v(z)dz.
\]
Clearly, we have ${\mathcal K}_{\omega}\boldsymbol u={\mathcal H}_{\omega}(\rho\boldsymbol u)$. For any bounded set $\mathcal V$ and arbitrary constant $\epsilon>0$, it is shown in \cite[Lemma 4.1]{LLW} that 
\begin{align*}
\|{\mathcal H}_{\omega}\|_{\mathcal{L}(\boldsymbol H_1^{-s}(\mathbb R^3),\boldsymbol H_{-1}^s(\mathbb R^3))}&\lesssim \omega^{-1+2s},\\
\|{\mathcal H}_{\omega}\|_{\mathcal{L}(\boldsymbol H_1^{-s}(\mathbb R^3),\boldsymbol L^\infty(\mathcal V))}&\lesssim \omega^{s+\epsilon+\frac12}.
\end{align*}
Next is to show that $\rho\boldsymbol u\in\boldsymbol H^{-s}_1(\mathbb R^3)$ for any $\boldsymbol u\in\boldsymbol H^s_{-1}(\mathbb R^3)$. 

For any $\boldsymbol u,\boldsymbol v\in\boldsymbol{\mathcal{S}}$, it holds
$
\langle\rho\boldsymbol u,\boldsymbol v\rangle=\langle\rho,\boldsymbol u\cdot\boldsymbol v\rangle.
$
Define a cutoff function $\chi\in C_0^\infty(\mathbb R^3)$, which has a bounded support $U$ with a locally Lipschitz boundary such that $D\subset U$ and $\chi(x)\equiv1$ if $x\in D$. It is clear to note that 
\begin{align*}
|\langle\rho\boldsymbol u,\boldsymbol v\rangle|&=|\langle\rho,(\chi\boldsymbol u)\cdot(\chi\boldsymbol v)\rangle|\\
&=|\langle(\mathcal I-\Delta)^{-\gamma}\rho,(\mathcal I-\Delta)^\gamma[(\chi\boldsymbol u)\cdot(\chi\boldsymbol v)]\rangle|\\
&\lesssim\|\rho\|_{W^{-\gamma,p}(D)}\|(\mathcal I-\Delta)^\gamma[(\chi\boldsymbol u)\cdot(\chi\boldsymbol v)]\|_{L^{p'}(\mathbb R^3)},
\end{align*}
where $p,p'\in(1,\infty)$ are conjugate indices satisfying $\frac1p+\frac1{p'}=1$. Using the fractional Leibniz principle leads to
\[
\|(\mathcal I-\Delta)^\gamma[(\chi\boldsymbol u)\cdot(\chi\boldsymbol v)]\|_{L^{p'}(\mathbb R^3)}\le\|\chi\boldsymbol u\|_{\boldsymbol L^2(U)}\|\chi\boldsymbol v\|_{\boldsymbol W^{\gamma,q}(U)}+\|\chi\boldsymbol v\|_{\boldsymbol L^2(U)}\|\chi\boldsymbol u\|_{\boldsymbol W^{\gamma,q}(U)},
\]
where $q$ satisfies $\frac1{p'}=\frac12+\frac1q$. Since $m\in(3-2s,3]$, there exists $\gamma\in(\frac{3-m}2,s]$ and $q\in(1,2]$ such that $\frac1q>\frac12-\frac{s-\gamma}3$, which implies from the Kondrachov compact embedding theorem that 
$\boldsymbol H^s(U)\hookrightarrow\boldsymbol W^{\gamma,q}(U)$. Hence,
\begin{align*}
|\langle\rho\boldsymbol u,\boldsymbol v\rangle|&\lesssim \|\rho\|_{W^{-\gamma,p}(D)}\|\chi\boldsymbol u\|_{\boldsymbol W^{\gamma,q}(U)}\|\chi\boldsymbol v\|_{\boldsymbol W^{\gamma,q}(U)}\\
&\lesssim \|\rho\|_{W^{-\gamma,p}(D)}\|\chi\boldsymbol u\|_{\boldsymbol H^s(U)}\|\chi\boldsymbol v\|_{\boldsymbol H^s(U)}\\
&\lesssim \|\rho\|_{W^{-\gamma,p}(D)}\|\boldsymbol u\|_{\boldsymbol H_{-1}^s(\mathbb R^3)}\|\boldsymbol v\|_{\boldsymbol H_{-1}^s(\mathbb R^3)},
\end{align*}
where in the last step we use the fact $\|\chi\boldsymbol u\|_{\boldsymbol H^s(U)}\lesssim\|\boldsymbol u\|_{\boldsymbol H_{-2}^s(\mathbb R^3)}\le\|\boldsymbol u\|_{\boldsymbol H_{-1}^s(\mathbb R^3)}$ (cf. \cite{CHL,LLM}). The proof is completed by noting  
\[
\|\rho\boldsymbol u\|_{\boldsymbol H^{-s}_1(\mathbb R^3)}:=\sup_{\boldsymbol v\in\boldsymbol H^s_{-1}(\mathbb R^3)}\frac{|\langle\rho\boldsymbol u,\boldsymbol v\rangle|}{\|\boldsymbol v\|_{\boldsymbol H_{-1}^s(\mathbb R^3)}}\lesssim\|\rho\|_{W^{-\gamma,p}(D)}\|\boldsymbol u\|_{\boldsymbol H_{-1}^s(\mathbb R^3)}
\]
and $\rho\in W^{\frac{m-3}2-\epsilon,p}(D)\subset W^{-\gamma,p}(D)$ according to Lemma \ref{lemmaadd3}.
\end{proof}

By the definition of $\boldsymbol{u}_j$ given in \eqref{b9}, we have  
\begin{eqnarray*}\label{b15}
(\mathcal I+{\mathcal K}_{\omega})\sum_{j=0}^N\boldsymbol{u}_j=\boldsymbol{u}_0+(-1)^N{\mathcal K}_{\omega}^{N+1}\boldsymbol{u}_0.
\end{eqnarray*}
For the leading term $\boldsymbol{u}_0=\boldsymbol{u}^{inc}$, a simple calculation yields 
\begin{eqnarray*}\label{b16}
\|\boldsymbol{u}_0\|_{\boldsymbol L^2(D)}\lesssim1,\quad \|\boldsymbol{u}_0\|_{\boldsymbol H^1(D)}\lesssim \omega.
\end{eqnarray*}
Using the interpolation inequality \cite{L3} leads to  
\begin{align*}
\|\boldsymbol{u}_0\|_{\boldsymbol H^s_{-1}(D)}&=\|(1+|\cdot|^2)^{-\frac{1}{2}}(\mathcal I-\Delta)^{\frac{s}{2}}\boldsymbol{u}_0\|_{\boldsymbol L^2(D)}\le \|(\mathcal I-\Delta)^{\frac{s}{2}}\boldsymbol{u}_0\|_{\boldsymbol L^2(D)}\\
&\lesssim\|\boldsymbol{u}_0\|_{\boldsymbol H^s(D)}\lesssim \|\boldsymbol{u}_0\|^{1-s}_{\boldsymbol L^2(D)}\|\boldsymbol{u}_0\|_{\boldsymbol H^1(D)}^s\lesssim \omega^s,
\end{align*}
which, together with Lemma \ref{lemma1}, gives that 
\begin{align*}
\|{\mathcal K}_{\omega}^{N+1}\boldsymbol{u}_0\|_{\boldsymbol H^s_{-1}(\mathbb R^3)}&\lesssim\|{\mathcal K}_{\omega}\|^{N}_{\mathcal{L}(\boldsymbol H^{s}_{-1}(\mathbb R^3),\boldsymbol H^{s}_{-1}(\mathbb R^3))}\|{\mathcal K}_{\omega}\|_{\mathcal{L}(\boldsymbol H^{s}_{-1}(D),\boldsymbol H^{s}_{-1}(\mathbb R^3))}\|\boldsymbol{u}_0\|_{\boldsymbol H^s_{-1}(D)}\\\label{b18}
&\lesssim \omega^{(-1+2s)(N+1)}\omega^s\to 0 \quad{\rm as}\;\; N\to\infty.
\end{align*}
Hence, we conclude 
\begin{eqnarray*}\label{b19}
(\mathcal I+{\mathcal K}_{\omega})\sum_{j=0}^N\boldsymbol{u}_j\to \boldsymbol{u}_0=(\mathcal I+{\mathcal K}_{\omega})\boldsymbol{u} \quad{\rm as}\;\; N\to\infty.
\end{eqnarray*}
Noting the invertibility of the operator $\mathcal I+{\mathcal K}_{\omega}$, we have 
\begin{eqnarray}\label{b20}
\boldsymbol{u}=\sum_{j=0}^{\infty}\boldsymbol{u}_j\quad{\rm in}\quad\boldsymbol H^s_{-1}(\mathbb R^3).
\end{eqnarray}
Moreover, for any bounded domain $U\subset\mathbb R^3$, it holds
\begin{align*}
&\|\boldsymbol{u}-\sum_{j=0}^N\boldsymbol{u}_j\|_{\boldsymbol L^{\infty}(U)}\lesssim \sum_{j=N+1}^{\infty}\|{\mathcal K}_{\omega}^j\boldsymbol{u}_0\|_{\boldsymbol L^{\infty}(U)}\\
&\lesssim\sum_{j=N+1}^{\infty}\|{\mathcal K}_{\omega}\|_{\mathcal{L}(\boldsymbol H^{s}_{-1}(\mathbb R^3),\boldsymbol L^{\infty}(U))}\|{\mathcal K}_{\omega}\|^{j-2}_{\mathcal{L}(\boldsymbol H^{s}_{-1}(\mathbb R^3),\boldsymbol H^{s}_{-1}(\mathbb R^3))}\|{\mathcal K}_{\omega}\|_{\mathcal{L}(\boldsymbol H^{s}_{-1}(D),\boldsymbol H^s_{-1}(\mathbb R^3))}\|\boldsymbol{u}_0\|_{\boldsymbol H^s_{-1}(D)}\\
&\lesssim\sum_{j=N+1}^{\infty}\omega^{s+\epsilon+\frac{1}{2}+(j-1)(-1+2s)+s}\to 0\quad{\rm as}\;\;N\to\infty,
\end{align*}
which implies that the convergence (\ref{b20}) also holds in $\boldsymbol L^{\infty}(U)$.

\section{The inverse scattering problem}

This section is to study the inverse problem, which aims to determine the microlocal strength $\phi$ of the random potential $\rho$ from the backscattered far-field pattern of the scattered wave. 

By (\ref{b20}), we rewrite the scattered wave as
\begin{eqnarray}\label{c1}
\boldsymbol{u}^{sc}(x) = \boldsymbol{u}_1(x)+\boldsymbol{u}_2(x)+\boldsymbol{b}(x),
\end{eqnarray}
where the residual $\boldsymbol{b}(x):=\sum_{j=3}^{\infty}\boldsymbol{u}_j(x)$.
Note that 
\begin{eqnarray}\label{c2}
\boldsymbol{u}_j(x)=-({\mathcal K}_{\omega}\boldsymbol{u}_{j-1})(x)=-\int_{\mathbb R^3}\boldsymbol{G}(x,z,\omega)\rho(z)\boldsymbol{u}_{j-1}(z)dz,
\end{eqnarray}
where the Green tensor $\boldsymbol{G}$ has the asymptotic behavior (cf. \cite[Section 2.2]{CS15})
\begin{eqnarray}\label{c3}
\boldsymbol{G}(x,z,\omega)&=&\frac{c_{\rm p}^2}{4\pi}\hat{x}\otimes\hat{x}\frac{e^{{\rm i}\kappa_{\rm p}|x|}}{|x|}e^{-{\rm i}\kappa_{\rm p}\hat{x}\cdot z}+\frac{c_{\rm s}^2}{4\pi}(\boldsymbol{I}-\hat{x}\otimes\hat{x})\frac{e^{{\rm i}\kappa_{\rm s}|x|}}{|x|}e^{-{\rm i}\kappa_{\rm s}\hat{x}\cdot z}+O(|x|^{-2}).
\end{eqnarray}
Here, the symbol $\hat{x}\otimes\hat{x}:=\hat{x}^{\top}\hat{x}\in\mathbb R^{3\times 3}$ is the tensor product. Substituting (\ref{c3}) into (\ref{c2}) leads to 
\begin{eqnarray}\label{c4}
\boldsymbol{u}_j(x)=\frac{e^{{\rm i}\kappa_{\rm p}|x|}}{|x|}\boldsymbol{u}^{\infty}_{j,\rm p}(\hat{x})+\frac{e^{{\rm i}\kappa_{\rm s}|x|}}{|x|}\boldsymbol{u}^{\infty}_{j,\rm s}(\hat{x})+O(|x|^{-2}),
\end{eqnarray}
where $\boldsymbol{u}_{j,\rm p}^{\infty}$ and $\boldsymbol{u}_{j,\rm s}^{\infty}$ are the compressional and shear far-field patterns of $\boldsymbol u_j$, respectively. A simple calculation from \eqref{c2} and \eqref{c4} gives 
\begin{equation}\label{c5}
\begin{aligned}
\boldsymbol{u}_{j,\rm p}^{\infty}(\hat{x}):=&-\frac{c_{\rm p}^2}{4\pi}\hat{x}\otimes\hat{x}\int_{\mathbb R^3}e^{-{\rm i}\kappa_{\rm p}\hat{x}\cdot z}\rho(z)\boldsymbol{u}_{j-1}(z)dz,\\
\boldsymbol{u}_{j,\rm s}^{\infty}(\hat{x}):=&-\frac{c_{\rm s}^2}{4\pi}(\boldsymbol{I}-\hat{x}\otimes\hat{x})\int_{\mathbb R^3}e^{-{\rm i}\kappa_{\rm s}\hat{x}\cdot z}\rho(z)\boldsymbol{u}_{j-1}(z)dz.
\end{aligned}
\end{equation}
Combining (\ref{a4}), (\ref{c1}) and (\ref{c4}), we get the following compressional and shear far-field patterns $\boldsymbol u_{\rm p}^\infty$ and $\boldsymbol u_{\rm s}^\infty$ of the scattered wave $\boldsymbol u^{sc}$:
\begin{equation}\label{c7}
\begin{split}
\boldsymbol{u}_{\rm p}^{\infty}(\hat{x})&=\boldsymbol{u}_{1,\rm p}^{\infty}(\hat{x})+\boldsymbol{u}_{2,\rm p}^{\infty}(\hat{x})+\boldsymbol{b}_{\rm p}^{\infty}(\hat{x}),\quad \boldsymbol{b}_{\rm p}^{\infty}(\hat{x}):=\sum_{j=3}^{\infty}\boldsymbol{u}_{j,\rm p}^{\infty}(\hat{x}),\\
\boldsymbol{u}_{\rm s}^{\infty}(\hat{x})&=\boldsymbol{u}_{1,\rm s}^{\infty}(\hat{x})+\boldsymbol{u}_{2,\rm s}^{\infty}(\hat{x})+\boldsymbol{b}_{\rm s}^{\infty}(\hat{x}),\quad \boldsymbol{b}_{\rm s}^{\infty}(\hat{x}):=\sum_{j=3}^{\infty}\boldsymbol{u}_{j,\rm s}^{\infty}(\hat{x}).
\end{split}
\end{equation}

As mentioned in the introduction, two types of incident plane waves are used as the illumination and two corresponding backscattered far-field patterns are measured as the data to reconstruct the strength $\phi$: one is the compressional plane wave $\boldsymbol{u}_0(x)=\boldsymbol{u}_{\rm p}^{inc}(x)=\theta e^{{\rm i}\kappa_{\rm p}x\cdot \theta}$ and the compressional far-field pattern $\boldsymbol u_{\rm p}^\infty(\hat x)$ is measured; another is the shear plane wave $\boldsymbol{u}_0(x)=\boldsymbol{u}_{\rm s}^{inc}(x)=\theta^{\perp}e^{{\rm i}\kappa_{\rm s}x\cdot \theta}$ and the shear far-field pattern $\boldsymbol u_{\rm s}^\infty(\hat x)$ is measured. 

To prove Theorem \ref{thm1}, we analyze separately the three terms in the far-field patterns \eqref{c7}: the first order far-field patterns $\boldsymbol{u}_{1,\rm p}^{\infty}$ and $\boldsymbol{u}_{1,\rm s}^{\infty}$, the second order far-field patterns $\boldsymbol{u}_{2,\rm p}^{\infty}$ and $\boldsymbol{u}_{2,\rm s}^{\infty}$, and the higher order far-field patterns $\boldsymbol{b}_{\rm p}^{\infty}$ and $\boldsymbol{b}_{\rm s}^{\infty}$. 

\subsection{The first order far-field patterns}\label{sec:3.1}

We begin with analyzing the first order backscattered far-field patterns by employing the Fourier analysis and ergodicity arguments. Below is the main result of this subsection.

\begin{theorem}\label{thm3}
Let the random potential $\rho$ satisfy Assumption \ref{as:rho}, $\boldsymbol{u}_{1,\rm p}^{\infty}$ and  $\boldsymbol{u}_{1,\rm s}^{\infty}$ be given by (\ref{c5}) with $\boldsymbol u_0=\boldsymbol u^{inc}_{\rm p}$ and $\boldsymbol u_0=\boldsymbol u^{inc}_{\rm s}$, respectively. For all $\theta\in {\mathbb S}^2$ and $\tau\geq 0$, it holds almost surely that 
\begin{align}\label{d1}
\lim_{Q\to\infty}\frac{1}{Q}\int_{Q}^{2Q}\omega^{m}\boldsymbol{u}_{1,\rm p}^{\infty}(-\theta,\omega,\theta)\cdot\overline{\boldsymbol{u}_{1, \rm p}^{\infty}(-\theta,\omega+\tau,\theta)}d\omega&= C_{\rm p}\hat{\phi}(2c_{\rm p}\tau \theta),\\\label{d2}
\lim_{Q\to\infty}\frac{1}{Q}\int_{Q}^{2Q}\omega^{m}\boldsymbol{u}_{1, \rm s}^{\infty}(-\theta,\omega,\theta)\cdot\overline{\boldsymbol{u}_{1, \rm s}^{\infty}(-\theta,\omega+\tau,\theta)}d\omega&= C_{\rm s}\hat{\phi}(2c_{\rm s}\tau \theta),
\end{align}
where $C_{\rm p}$ and $C_{\rm s}$ are constants defined in Theorem \ref{thm1}.
\end{theorem}

The proof of Theorem \ref{thm3} is left at the end of this subsection. The following lemmas are useful for the proof of Theorem \ref{thm3}.

\begin{lemma}\label{lemma2}
Under assumptions in Theorem \ref{thm3}, for all $\theta\in {\mathbb S}^2$ and $\tau\geq 0$, it holds that 
\begin{align}\label{x1}
\lim_{Q\to\infty}\frac{1}{Q}\int_{Q}^{2Q}\omega^{m}{\mathbb E}\left[\boldsymbol{u}_{1,\rm p}^{\infty}(-\theta,\omega,\theta)\cdot\overline{\boldsymbol{u}_{1, \rm p}^{\infty}(-\theta,\omega+\tau,\theta)}\right]d\omega&= C_{\rm p}\hat{\phi}(2c_{\rm p}\tau \theta),\\\label{x2}
\lim_{Q\to\infty}\frac{1}{Q}\int_{Q}^{2Q}\omega^{m}{\mathbb E}\left[\boldsymbol{u}_{1, \rm s}^{\infty}(-\theta,\omega,\theta)\cdot\overline{\boldsymbol{u}_{1, \rm s}^{\infty}(-\theta,\omega+\tau,\theta)}\right]d\omega&= C_{\rm s}\hat{\phi}(2c_{\rm s}\tau \theta),
\end{align}
where $C_{\rm p}$ and $C_{\rm s}$ are constants defined in Theorem \ref{thm1}.
\end{lemma}

\begin{proof}
Using (\ref{a2}), (\ref{c5}), and noting $(\theta\otimes \theta)\theta=\theta$ and $(\theta\otimes \theta)\theta^{\perp}=0$, we obtain
\begin{eqnarray*}\label{d4}
\boldsymbol{u}^{\infty}_{1,\rm p}(-\theta,\omega,\theta)=-\frac{c_{\rm p}^2}{4\pi}\theta\otimes \theta\int_{\mathbb R^3}e^{{\rm i}\kappa_{\rm p}\theta\cdot z}\rho(z) \theta e^{{\rm i}\kappa_{\rm p}\theta\cdot z}dz
=-\frac{c_{\rm p}^2}{4\pi} \theta\int_{\mathbb R^3}e^{2{\rm i}\kappa_{\rm p}\theta\cdot z}\rho(z)dz
\end{eqnarray*}
and 
\begin{eqnarray*}\label{d5}
\boldsymbol{u}^{\infty}_{1,\rm s}(-\theta,\omega,\theta)=-\frac{c_{\rm s}^2}{4\pi}(\boldsymbol{I}-\theta\otimes \theta)\int_{\mathbb R^3}e^{{\rm i}\kappa_{\rm s}\theta\cdot z}\rho(z) \theta^{\perp}e^{{\rm i}\kappa_{\rm s}\theta\cdot z}dz
=-\frac{c_{\rm s}^2}{4\pi} \theta^{\perp}\int_{\mathbb R^3}e^{2{\rm i}\kappa_{\rm s}\theta\cdot z}\rho(z)dz.
\end{eqnarray*}
It suffices to show \eqref{x1} since the proof is similar for \eqref{x2}. 

We have for $\omega_1,\omega_2\geq 1$ that 
\begin{align}\label{d6}
{\mathbb E}\left[\boldsymbol{u}_{1,\rm p}^{\infty}(-\theta,\omega_1,\theta)\cdot\overline{\boldsymbol{u}_{1,\rm p}^{\infty}(-\theta,\omega_2,\theta)}\right]
&=\frac{c_{\rm p}^4}{16\pi^2}\int_{\mathbb R^3}\int_{\mathbb R^3}e^{2{\rm i}c_{\rm p}\omega_1 \theta\cdot y}e^{-2{\rm i}c_{\rm p}\omega_2 \theta\cdot z}{\mathbb E}\left(\rho(y)\rho(z)\right)dydz\notag\\
&=\frac{c_{\rm p}^4}{16\pi^2}\int_{\mathbb R^3}\int_{\mathbb R^3}e^{2{\rm i}c_{\rm p}\theta\cdot (\omega_1 y-\omega_2 z)}K_\rho(y,z)dydz,
\end{align}
where $K_\rho\in \mathcal{D'}({\mathbb R^3}\times{\mathbb R^3}; \mathbb R)$ is the symmetric covariance kernel of $\rho$ satisfying
\begin{eqnarray*}\label{d7}
\langle {\mathcal C}_\rho\varphi,\psi\rangle=\mathbb E[\langle\rho,\varphi\rangle\langle\rho,\psi\rangle]=\int_{\mathbb R^3}\int_{\mathbb R^3}K_\rho(y,z)\varphi(y)\overline{\psi(z)}dydz\quad\forall~\varphi,\psi\in\mathcal D.
\end{eqnarray*}
Let $s_\rho\in\mathcal{S}^{-m}(\mathbb R^3\times\mathbb R^3)$ be the symbol of the covariance operator ${\mathcal C}_{\rho}$ satisfying
\[
({\mathcal C}_\rho\varphi)(x)=\frac1{(2\pi)^3}\int_{\mathbb R^3}e^{{\rm i}x\cdot\xi}s_{\rho}(x,\xi)\hat{\varphi}(\xi)d\xi\quad\forall~\varphi\in\mathcal D,
\] 
where $\mathcal{S}^{-m}(\mathbb R^3\times\mathbb R^3)$ is defined by
\begin{eqnarray*}
\mathcal{S}^{-m}(\mathbb R^3\times\mathbb R^3):&=&\Big\{s(x,\xi)\in C^{\infty}(\mathbb R^3\times\mathbb R^3):|\partial_{\xi}^{\gamma_1}\partial_x^{\gamma_2}s(x,\xi)|\leq C(\gamma_1,\gamma_2)(1+|\xi|)^{-m-|\gamma_1|}\Big\}
\end{eqnarray*}
with $\gamma_1$ and $\gamma_2$ being any multiple indices and $|\gamma_1|$ denoting the sum of its components. 
A simple calculation gives the oscillatory integral form of $K_\rho$ (cf. \cite{LW1}):
\begin{equation}\label{dK}
K_\rho(y,z)=\frac1{(2\pi)^3}\int_{\mathbb R^3}e^{{\rm i}(z-y)\cdot\xi}s_\rho(z,\xi)d\xi.
\end{equation}

According to Assumption \ref{as:rho}, we have $s_\rho(x,\xi)=\phi(x)|\xi|^{-m}+a(x,\xi)$ where $a\in\mathcal{S}^{-m-1}(\mathbb R^3\times\mathbb R^3)$, and ${\rm supp} K_\rho\subset D\times D$. Substituting \eqref{dK} into (\ref{d6}) yields
\begin{align}\nonumber
&{\mathbb E}\left[\boldsymbol{u}_{1,\rm p}^{\infty}(-\theta,\omega_1,\theta)\cdot\overline{\boldsymbol{u}_{1,\rm p}^{\infty}(-\theta,\omega_2,\theta)}\right]\\\nonumber
&=\frac{c_{\rm p}^4}{16\pi^2}\int_{\mathbb R^3}\int_{\mathbb R^3}e^{2{\rm i}c_{\rm p}\theta\cdot (\omega_1 y-\omega_2 z)}\left[\frac1{(2\pi)^3}\int_{\mathbb R^3}e^{{\rm i}(z-y)\cdot\xi}s_\rho(z,\xi)d\xi\right]dydz\\\nonumber
&=\frac{c_{\rm p}^4}{16\pi^2}\int_{\mathbb R^3}\int_{\mathbb R^3}e^{-2{\rm i}c_{\rm p}\omega_2\theta\cdot z+{\rm i}z\cdot\xi}s_\rho(z,\xi)\delta(\xi-2c_{\rm p}\omega_1\theta)d\xi dz\\\nonumber
&=\frac{c_{\rm p}^4}{16\pi^2}\int_{D}s_{\rho}(z,2c_{\rm p}\omega_1 \theta)e^{2{\rm i}c_{\rm p}(\omega_1-\omega_2)\theta\cdot z}dz\\\nonumber
&=\frac{c_{\rm p}^4}{16\pi^2}\left[\int_{D}\phi(z)|2c_{\rm p}\omega_1\theta|^{-m}e^{2{\rm i}c_{\rm p}(\omega_1-\omega_2)\theta\cdot z}dz+\int_{D}a(z,2c_{\rm p}\omega_1\theta)e^{2{\rm i}c_{\rm p}(\omega_1-\omega_2)\theta\cdot z}dz\right]\\\label{d10}
&=C_{\rm p}\hat{\phi}(2c_{\rm p}(\omega_2-\omega_1)\theta)\omega_1^{-m}+O(\omega_1^{-m-1}).
\end{align}
Letting $\omega_1=\omega$ and $\omega_2=\omega+\tau$ in (\ref{d10}) gives  
\begin{eqnarray*}\label{d11}
{\mathbb E}\left[\boldsymbol{u}_{1,\rm p}^{\infty}(-\theta,\omega,\theta)\cdot\overline{\boldsymbol{u}_{1,\rm p}^{\infty}(-\theta,\omega+\tau,\theta)}\right]=C_{\rm p}\hat{\phi}(2c_{\rm p}\tau \theta)\omega^{-m}+O(\omega^{-m-1}),
\end{eqnarray*}
which implies \eqref{x1} and completes the proof. 
\end{proof}

\begin{lemma}\label{lemma3}
Under assumptions in Theorem \ref{thm3}, it holds for all $\theta\in {\mathbb S}^2$,  $\omega_1,\omega_2\geq 1$ and $N\in{\mathbb N}$ that
\begin{align}\label{d15}
\left|{\mathbb E}\left[\boldsymbol{u}_{1,\rm p}^{\infty}(-\theta,\omega_1,\theta)\cdot\overline{\boldsymbol{u}_{1, \rm p}^{\infty}(-\theta,\omega_2,\theta)}\right]\right|&\lesssim \omega_1^{-m}(1+|\omega_1-\omega_2|)^{-N},\\\label{d16}
\left|{\mathbb E}\left[\boldsymbol{u}_{1,\rm s}^{\infty}(-\theta,\omega_1,\theta)\cdot\overline{\boldsymbol{u}_{1, \rm s}^{\infty}(-\theta,\omega_2,\theta)}\right]\right|&\lesssim \omega_1^{-m}(1+|\omega_1-\omega_2|)^{-N},\\\label{d18}
\left|{\mathbb E}\left[\boldsymbol{u}_{1,\rm p}^{\infty}(-\theta,\omega_1,\theta)\cdot\boldsymbol{u}_{1, \rm p}^{\infty}(-\theta,\omega_2,\theta)\right]\right|&\lesssim \omega_1^{-m}(1+\omega_1+\omega_2)^{-N},\\\label{d19}
\left|{\mathbb E}\left[\boldsymbol{u}_{1,\rm s}^{\infty}(-\theta,\omega_1,\theta)\cdot\boldsymbol{u}_{1, \rm s}^{\infty}(-\theta,\omega_2,\theta)\right]\right|&\lesssim \omega_1^{-m}(1+\omega_1+\omega_2)^{-N}.
\end{align}
\end{lemma}
\begin{proof}
For the case $|\omega_1-\omega_2|<1$, it follows from (\ref{d10}) that 
\begin{align*}
&\left|{\mathbb E}\left[\boldsymbol{u}_{1,\rm p}^{\infty}(-\theta,\omega_1,\theta)\cdot\overline{\boldsymbol{u}_{1, \rm p}^{\infty}(-\theta,\omega_2,\theta)}\right]\right|\\
&\le \frac{c_{\rm p}^4}{16\pi^2}\int_D\left|s_{\rho}(z,2c_{\rm p}\omega_1\theta)e^{2{\rm i}c_{\rm p}(\omega_1-\omega_2)\theta\cdot z}\right|dz\\
&\lesssim (1+\omega_1)^{-m}\lesssim 2^N(1+\omega_1)^{-m}(1+|\omega_1-\omega_2|)^{-N}\\
&\lesssim \omega_1^{-m}(1+|\omega_1-\omega_2|)^{-N},
\end{align*}
where we use the fact $s_{\rho}\in {\mathcal S}^{-m}$ and hence $|s_\rho(z,2c_{\rm p}\omega_1\theta)|\lesssim(1+|\omega_1|)^{-m}$.

For the case $|\omega_1-\omega_2|\geq 1$, denoting $z=(z_1,z_2,z_3)^\top$ and $\theta=(\theta_1,\theta_2,\theta_3)^\top$, we obtain from (\ref{d10}) and the integration by parts that 
\begin{align}\nonumber
&{\mathbb E}\left[\boldsymbol{u}_{1,\rm p}^{\infty}(-\theta,\omega_1,\theta)\cdot\overline{\boldsymbol{u}_{1, \rm p}^{\infty}(-\theta,\omega_2,\theta)}\right]\\\nonumber
&=\frac{c_{\rm p}^4}{16\pi^2}\int_{D}s_{\rho}(z,2c_{\rm p}\omega_1\theta)e^{2{\rm i}c_{\rm p}(\omega_1-\omega_2)\theta\cdot z}dz\\\nonumber
&=\frac{c_{\rm p}^4}{16\pi^2}\frac{1}{2{\rm i}c_{\rm p}(\omega_1-\omega_2)\theta_1}\int_Ds_{\rho}(z,2c_{\rm p}\omega_1\theta)\\\nonumber
&\qquad\qquad \times e^{2{\rm i}c_{\rm p}(\omega_1-\omega_2)(\theta_2z_2+\theta_3z_3)}de^{2{\rm i}c_{\rm p}(\omega_1-\omega_2)\theta_1z_1}dz_2dz_3\\\nonumber
&=-\frac{c_{\rm p}^4}{16\pi^2}\frac{1}{2{\rm i}c_{\rm p}(\omega_1-\omega_2)\theta_1}\int_{D}\partial_{z_1}s_{\rho}(z,2c_{\rm p}\omega_1\theta)e^{2{\rm i}c_{\rm p}(\omega_1-\omega_2)\theta\cdot z}dz\\\label{d21}
&=(-1)^N\frac{c_{\rm p}^4}{16\pi^2}\frac{1}{(2{\rm i}c_{\rm p}(\omega_1-\omega_2)\theta_1)^N}\int_{D}\partial^N_{z_1}s_{\rho}(z,2c_{\rm p}\omega_1\theta)e^{2{\rm i}c_{\rm p}(\omega_1-\omega_2)\theta\cdot z}dz.
\end{align}
Since $s_{\rho}\in {\mathcal S}^{-m}$, we have
\begin{eqnarray*}\label{d22}
\left|\partial^N_{z_1}s_{\rho}(z,2c_{\rm p}\omega_1\theta)\right|\lesssim (1+\omega_1)^{-m}.
\end{eqnarray*}
Combining the above estimates leads to 
\begin{align}\nonumber
\left|{\mathbb E}\left[\boldsymbol{u}_{1,\rm p}^{\infty}(-\theta,\omega_1,\theta)\cdot\overline{\boldsymbol{u}_{1, \rm p}^{\infty}(-\theta,\omega_2,\theta)}\right]\right|
&\lesssim (1+\omega_1)^{-m}\left|\omega_1-\omega_2\right|^{-N}\\\nonumber
&\lesssim \left(1+\frac{1}{|\omega_1-\omega_2|}\right)^N\omega_1^{-m}(1+|\omega_1-\omega_2|)^{-N}\\\nonumber
&\lesssim 2^N\omega_1^{-m}(1+|\omega_1-\omega_2|)^{-N}\\\label{d23}
&\lesssim \omega_1^{-m}(1+|\omega_1-\omega_2|)^{-N},
\end{align}
which shows (\ref{d15}).

The inequality (\ref{d16}) can be obtained by following the same procedure; the inequalities (\ref{d18}) and (\ref{d19}) can be proved similarly by replacing $\omega_2$ by $-\omega_2$ in (\ref{d21}) and (\ref{d23}), respectively. 
\end{proof}

The following two lemmas help to replace the results in the expectation sense stated in Lemma \ref{lemma2} with the ones in the almost surely sense given in Theorem \ref{thm3}. The proof of Lemma \ref{lemmaadd2} can be found in \cite{CHL}. 

\begin{lemma}\label{lemmaadd2}
Let $X$ and $Y$ be two random variables such that the pair $(X,Y)$ is a Gaussian random vector. If $\mathbb E[X]=\mathbb E[Y]=0$, then
\begin{eqnarray*}
\mathbb E\left[(X^2-\mathbb EX^2)(Y^2-\mathbb EY^2)\right]=2\left(\mathbb E[XY]\right)^2.
\end{eqnarray*}
\end{lemma}

\begin{lemma}\label{lemmaadd1}
Let $\{X_t\}_{t\ge0}$ be a real-valued centered stochastic process with continuous paths and $\mathbb E[X_t]=0$. Assume that for some positive constants $\eta$ and $\sigma$, it holds
\begin{eqnarray*}
\left|\mathbb E[X_tX_{t+r}]\right|\lesssim(1+|r-\eta|)^{-\sigma}\quad\forall~t,r\ge0.
\end{eqnarray*}
Then
\begin{eqnarray*}
\lim_{T\to\infty}\frac{1}{T}\int_0^TX_tdt=\lim_{T\to\infty}\frac{1}{T}\int_T^{2T}X_tdt=0\quad a.s.
\end{eqnarray*}
\end{lemma}
\begin{proof}
Without loss of generality, we assume that $\sigma\in(0,1)$. If the condition in Lemma \ref{lemmaadd1} holds for $\sigma'\ge1$, then we can always find some $\sigma\in(0,1)$ such that
\[
\left|\mathbb E[X_tX_{t+r}]\right|\lesssim(1+|r-\eta|)^{-\sigma'}<(1+|r-\eta|)^{-\sigma}.
\] 

For $T$ being large enough such that $T>\eta$, we have
\begin{align*}
&{\mathbb E}\left|\frac{1}{T}\int_0^TX_tdt\right|^2=\frac{1}{T^2}\int_0^T\int_0^T{\mathbb E}[X_tX_u]dtdu\\
&\lesssim \frac{1}{T^2}\int_0^T\int_0^T\left(1+\left||t-u|-\eta\right|\right)^{-\sigma}dtdu\\
&=\frac{1}{T^2}\int_0^T\left[\int_0^u\left(1+|u-t-\eta|\right)^{-\sigma}dt+\int_u^T(1+|t-u-\eta|)^{-\sigma}dt\right]du\\
&=\frac{1}{T^2}\int_0^T\bigg[\int_0^{(u-\eta)\vee0}(1+u-\eta-t)^{-\sigma}dt+\int_{(u-\eta)\vee0}^u(1+t-u+\eta)^{-\sigma}dt\\
&\quad +\int_u^{(u+\eta)\wedge T}(1+u+\eta-t)^{-\sigma}dt+\int_{(u+\eta)\wedge T}^T(1+t-u-\eta)^{-\sigma}dt\bigg]du\\
&=\frac{1}{(1-\sigma)T^2}\int_0^T\Big[2(1+\eta)^{1-\sigma}+(1+u-\eta)^{1-\sigma}+(1+T-u-\eta)^{1-\sigma}\\
&\quad -(1+u-\eta-(u-\eta)\vee0)^{1-\sigma}-(1-u+\eta+(u-\eta)\vee0)^{1-\sigma}\\
&\quad -(1+u+\eta-(u+\eta)\wedge T)^{1-\sigma}-(1-u-\eta+(u+\eta)\wedge T)^{1-\sigma}\Big]du\\
&=\frac{2(1+\eta)^{1-\sigma}}{(1-\sigma)T}+\frac{2\left[(1+T-\eta)^{2-\sigma}-(1-\eta)^{2-\sigma}\right]}{(2-\sigma)(1-\sigma)T^2}\\
&\quad -\frac{2}{(1-\sigma)T^2}\left[\frac{(1+\eta)^{2-\sigma}-(1-\eta)^{2-\sigma}}{2-\sigma}+2(T-\eta)\right],
\end{align*}
where we use the notations $a\vee b:=\max\{a,b\}$ and $a\wedge b:=\min\{a,b\}$.

It follows from Fatou's lemma that 
\begin{eqnarray*}
\mathbb E\left|\lim_{T\to\infty}\frac{1}{T}\int_0^TX_tdt\right|^2\le\varliminf_{T\to\infty}{\mathbb E}\left|\frac{1}{T}\int_0^TX_tdt\right|^2=0.
\end{eqnarray*}
Hence
\begin{eqnarray*}
\lim_{T\to\infty}\frac{1}{T}\int_0^TX_tdt=0\quad a.s.
\end{eqnarray*}
and
\begin{eqnarray*}
\lim_{T\to\infty}\frac{1}{T}\int_T^{2T}X_tdt=\lim_{T\to\infty}\left[\frac{1}{T}\int_0^{2T}X_tdt-\frac{1}{T}\int_0^TX_tdt\right]=0\quad a.s.,
\end{eqnarray*}
which complete the proof. 
\end{proof}

Now we are in the position to prove Theorem \ref{thm3}.

\begin{proof}[Proof of Theorem \ref{thm3}] 
By Lemma \ref{lemma2}, to prove (\ref{d1}), it suffices to show that
\begin{align*}
&\lim_{Q\to\infty}\frac{1}{Q}\int_Q^{2Q}\omega^m\boldsymbol{u}_{1,\rm p}^{\infty}(-\theta,\omega,\theta)\cdot\overline{\boldsymbol{u}_{1,\rm p}^{\infty}(-\theta,\omega+\tau,\theta)}d\omega\\
=&\lim_{Q\to\infty}\frac{1}{Q}\int_Q^{2Q}\omega^m{\mathbb E}\left[\boldsymbol{u}_{1,\rm p}^{\infty}(-\theta,\omega,\theta)\cdot\overline{\boldsymbol{u}_{1,\rm p}^{\infty}(-\theta,\omega+\tau,\theta)}\right]d\omega,
\end{align*}
or equivalently,
\begin{eqnarray}\label{d25}
\sum_{j=1,2,3}\lim_{Q\to\infty}\frac1Q\int_Q^{2Q}\omega^m\left(u_j(\omega)\overline{u_j(\omega+\tau)}-\mathbb E\left[u_j(\omega)\overline{u_j(\omega+\tau)}\right]\right)d\omega=0,
\end{eqnarray}
where $\boldsymbol u_{1,\rm p}^\infty(-\theta,\omega,\theta)=(u_1(\omega),u_2(\omega),u_3(\omega))^\top$.

Denote by $U_j(\omega)$ and $ V_j(\omega)$ the real and imaginary parts of $u_j(\omega)$, respectively, which reads
\begin{eqnarray}\label{d26}
u_j(\omega)=U_j(\omega)+{\rm i}V_j(\omega),\quad j=1,2,3.
\end{eqnarray}
It then leads to 
\begin{align*}
2u_j(\omega)\overline{u_j(\omega+\tau)}
&=2(U_j(\omega)+{\rm i}V_j(\omega))(U_j(\omega+\tau)-{\rm i}V_j(\omega+\tau))\\
&=(1+{\rm i})\left[U_j^2(\omega)+U_j^2(\omega+\tau)+V_j^2(\omega)+V_j^2(\omega+\tau)\right]\\
&\quad -(U_j(\omega)-U_j(\omega+\tau))^2
-(V_j(\omega)-V_j(\omega+\tau))^2\\
&\quad -{\rm i}(U_j(\omega)+V_j(\omega+\tau))^2-{\rm i}(V_j(\omega)-U_j(\omega+\tau))^2.
\end{align*}
For simplicity, let $W_{\omega}$ be any random variable in the set $\Gamma:=\{U_j(\omega),U_j(\omega+\tau),V_j(\omega),V_j(\omega+\tau),U_j(\omega)-U_j(\omega+\tau),V_j(\omega)-V_j(\omega+\tau),U_j(\omega)+ V_j(\omega+\tau),V_j(\omega)-U_j(\omega+\tau)\}_{j=1,2,3}$. 

Then it only requires to show
\begin{eqnarray}\label{d28}
\lim_{Q\to\infty}\frac{1}{Q}\int_Q^{2Q}\omega^m\left(W_{\omega}^2-{\mathbb E}W_{\omega}^2\right)d\omega=0,
\end{eqnarray}
which indicates (\ref{d25}). Using Lemmas \ref{lemmaadd1} and \ref{lemmaadd2} and noting that $W_\omega$ is Gaussian since $\rho$ is Gaussian, to get (\ref{d28}), we need to show that for any $W_\omega\in\Gamma$, there exist positive constants $\eta$ and $\sigma$ such that
\begin{align}\label{d29}
&\left|\mathbb E\left[\omega^m\left(W_{\omega}^2-\mathbb EW_{\omega}^2\right)(\omega+r)^m\left(W_{\omega+r}^2-\mathbb EW_{\omega+r}^2\right)\right]\right|\notag\\
&=2\left(\mathbb E\left[\omega^{\frac m2}(\omega+r)^{\frac{m}{2}}W_{\omega}W_{\omega+r}\right]\right)^2
\lesssim(1+|r-\eta|)^{-\sigma}\quad \forall~\omega\ge1,\,r\ge0.
\end{align} 

It follows from (\ref{d26}) that 
\begin{align*}
U_j(\omega)=\frac{1}{2}\left[u_j(\omega)+\overline{u_j(\omega)}\right],\quad 
V_j(\omega)=\frac{1}{2{\rm i}}\left[u_j(\omega)-\overline{u_j(\omega)}\right],
\end{align*}
which give
\begin{align*}
U_j(\omega_1)U_j(\omega_2)&=\frac{1}{4}\left[u_j(\omega_1)+\overline{u_j(\omega_1)}\right]\left[u_j(\omega_2)+\overline{u_j(\omega_2)}\right],\\
V_j(\omega_1)V_j(\omega_2)&=-\frac{1}{4}\left[u_j(\omega_1)-\overline{u_j(\omega_1)}\right]\left[u_j(\omega_2)-\overline{u_j(\omega_2)}\right],\\
U_j(\omega_1)V_j(\omega_2)&=\frac{1}{4{\rm i}}\left[u_j(\omega_1)+\overline{u_j(\omega_1)}\right]\left[u_j(\omega_2)-\overline{u_j(\omega_2)}\right].
\end{align*}
Using the same procedure as that in Lemma \ref{lemma3} yields 
\begin{align*}
\left|\mathbb E[U_j(\omega_1)U_j(\omega_2)]\right|&\lesssim\omega_1^{-m}(1+|\omega_1-\omega_2|)^{-N},\\
\left|\mathbb E[V_j(\omega_1)V_j(\omega_2)]\right|&\lesssim\omega_1^{-m}(1+|\omega_1-\omega_2|)^{-N},\\
\left|\mathbb E[U_j(\omega_1)V_j(\omega_2)]\right|&\lesssim\omega_1^{-m}(1+|\omega_1-\omega_2|)^{-N},
\end{align*}
which leads to 
\begin{eqnarray*}
\left|\mathbb E\left[\omega^{\frac{m}{2}}(\omega+r)^{\frac{m}{2}}U_j(\omega)U_j(\omega+r)\right]\right|&\lesssim& \Big(\frac{1+r/\omega}{1+r}\Big)^{\frac{m}{2}}(1+r)^{\frac{m}{2}-N}
\lesssim (1+r)^{\frac{m}{2}-N}
\end{eqnarray*}
for any $\omega\ge1$ and $N\in\mathbb N$.
Similarly, we may conclude that 
\begin{eqnarray}\label{d40}
\left|\mathbb E\left[\omega^{\frac{m}{2}}(\omega+r)^{\frac{m}{2}}W_{\omega}W_{\omega+r}\right]\right|\lesssim(1+r)^{\frac{m}{2}-N}
\end{eqnarray}
holds for $W_{\omega}\in\{U_j(\omega),V_j(\omega), U_j(\omega+\tau),V_j(\omega+\tau)\}_{j=1,2,3}$. 

For the case $W_{\omega}=U_j(\omega)-U_j(\omega+\tau)$, we have from Lemma \ref{lemma3} that 
\begin{align*}
&\left|\mathbb E\left[(U_j(\omega)-U_j(\omega+\tau))(U_j(\omega+r)-U_j(\omega+r+\tau))\right]\right|\\
&\lesssim\left|\mathbb E[U_j(\omega)U_j(\omega+r)]\right|+\left|\mathbb E[U_j(\omega)U_j(\omega+r+\tau)]\right|\\
&\quad +\left|\mathbb E[U_j(\omega+\tau)U_j(\omega+r)]\right|+\left|\mathbb E[U_j(\omega+\tau)U_j(\omega+r+\tau)]\right|\\
&\lesssim \omega^{-m}(1+r)^{-N}+ \omega^{-m}(1+r+\tau)^{-N}+ (\omega+\tau)^{-m}(1+|r-\tau|)^{-N}+ (\omega+\tau)^{-m}(1+r)^{-N}\\
&\lesssim \omega^{-m}(1+r)^{-N}+ (\omega+\tau)^{-m}(1+|r-\tau|)^{-N}+ (\omega+\tau)^{-m}(1+r)^{-N}.
\end{align*}
Hence,
\begin{align*}
&\left|\mathbb E\left[\omega^{\frac{m}{2}}(\omega+r)^{\frac{m}{2}}(U_j(\omega)-U_j(\omega+\tau))(U_j(\omega+r)-U_j(\omega+r+\tau))\right]\right|\\
&\lesssim (1+r)^{\frac{m}{2}-N}+\Big(\frac{\omega+r}{\omega+\tau}\Big)^{\frac{m}{2}}(1+|r-\tau|)^{-N}+\Big(\frac{\omega+r}{\omega+\tau}\Big)^{\frac{m}{2}}(1+r)^{-N}\\
&\lesssim (1+r)^{\frac{m}{2}-N}+(1+|r-\tau|)^{\frac{m}{2}-N}.
\end{align*}
Similarly, we may show that the inequality 
\begin{eqnarray}\label{d43}
\left|\mathbb E\left[\omega^{\frac{m}{2}}(\omega+r)^{\frac{m}{2}}W_{\omega}W_{\omega+r}\right]\right|\lesssim (1+r)^{\frac{m}{2}-N}+(1+|r-\tau|)^{\frac{m}{2}-N}
\end{eqnarray}
holds for $W_{\omega}\in\{U_j(\omega)-U_j(\omega+\tau),V_j(\omega)-V_j(\omega+\tau), U_j(\omega)+V_j(\omega+\tau),V_j(\omega)-U_j(\omega+\tau)\}_{j=1,2,3}$. 

Combining (\ref{d40}) and (\ref{d43}), we get that (\ref{d29}) holds for all $W_{\omega}\in\Gamma$, which completes the proof of (\ref{d1}). The proof of (\ref{d2}) is analogous to the proof of (\ref{d1}), and is omitted here. 
\end{proof}

\subsection{The second order far-field patterns}\label{sec:3.2}

In this subsection, we show that the contribution of the second order backscattered far-field pattern can be ignored. 
According to (\ref{c2}) and (\ref{c5}), the far-field patterns $\boldsymbol u_{2,\rm p}^\infty$ and $\boldsymbol u_{2,\rm s}^\infty$ associated with incident waves $\boldsymbol u_0=\boldsymbol u_{\rm p}^{inc}$ and $\boldsymbol u_0=\boldsymbol u_{\rm s}^{inc}$, respectively, admit the following forms: 
\begin{equation}\label{e2}
\begin{aligned}
\boldsymbol{u}_{2,\rm p}^{\infty}(-\theta,\omega,\theta)&=\frac{c_{\rm p}^2}{4\pi}\theta\otimes \theta\int_{\mathbb R^3}\int_{\mathbb R^3}\rho(y)\rho(z)\boldsymbol{G}(y,z,\omega)\theta e^{{\rm i}\kappa_{\rm p}\theta\cdot (y+z)}dzdy,\\
\boldsymbol{u}_{2,\rm s}^{\infty}(-\theta,\omega,\theta)&=\frac{c_{\rm s}^2}{4\pi}(\boldsymbol{I}-\theta\otimes \theta)\int_{\mathbb R^3}\int_{\mathbb R^3}\rho(y)\rho(z)\boldsymbol{G}(y,z,\omega)\theta^{\perp} e^{{\rm i}\kappa_{\rm s}\theta\cdot (y+z)}dzdy.
\end{aligned}
\end{equation}

The main result of this subsection is stated in the following theorem. 

\begin{theorem}\label{thm4}
Let the random potential $\rho$ satisfy Assumption \ref{as:rho}, $\boldsymbol{u}_{2,\rm p}^{\infty}$ and $\boldsymbol{u}_{2,\rm s}^{\infty}$ be given by (\ref{e2}).  
For all $\theta\in{\mathbb S^2}$, it holds almost surely that 
\begin{align*}
\lim_{Q\to\infty}\frac{1}{Q}\int_{Q}^{2Q}\omega^m\left|\boldsymbol{u}_{2,\rm p}^{\infty}(-\theta,\omega,\theta)\right|^2d\omega &=0,\\
\lim_{Q\to\infty}\frac{1}{Q}\int_{Q}^{2Q}\omega^m\left|\boldsymbol{u}_{2,\rm s}^{\infty}(-\theta,\omega,\theta)\right|^2d\omega&=0.
\end{align*}
\end{theorem}

The Green tensor $\boldsymbol{G}$ given in (\ref{b2}) can be split into three parts
\begin{eqnarray}\label{e7}
\boldsymbol{G}(y,z,\omega)=\boldsymbol{G}_1(y,z,\omega)+\boldsymbol{G}_2(y,z,\omega)+\boldsymbol{G}_3(y,z,\omega),
\end{eqnarray}
where
\begin{align*}
\boldsymbol{G}_1(y,z,\omega)&=\frac{c_{\rm s}^2}{4\pi}\frac{e^{{\rm i}\kappa_{\rm s}|y-z|}}{|y-z|}\boldsymbol{I},\\
\boldsymbol{G}_2(y,z,\omega)&=\frac{c_{\rm p}^2e^{{\rm i}\kappa_{\rm p}|y-z|}-c_{\rm s}^2e^{{\rm i}\kappa_{\rm s}|y-z|}}{4\pi |y-z|^3}(y-z)\otimes (y-z),\\
\boldsymbol{G}_3(y,z,\omega)&=\omega^{-2}\frac{\beta(y,z,\omega)}{4\pi |y-z|^5}\left[|y-z|^2\boldsymbol{I}-3(y-z)\otimes (y-z)\right].
\end{align*}
Here 
\begin{eqnarray*}
\beta(y,z,\omega):=e^{{\rm i}\kappa_{\rm s}|y-z|}({\rm i}\kappa_{\rm s}|y-z|-1)-e^{{\rm i}\kappa_{\rm p}|y-z|}({\rm i}\kappa_{\rm p}|y-z|-1).
\end{eqnarray*}
Substituting (\ref{e7}) into (\ref{e2}), we can see that $\boldsymbol{u}_{2,\rm p}^{\infty}$ and $\boldsymbol{u}_{2,\rm s}^{\infty}$ also consist of three parts corresponding to $\boldsymbol{G}_1$, $\boldsymbol{G}_2$ and $\boldsymbol{G}_3$. The components in the first and second parts are linear combinations of 
\begin{equation}\label{e11}
\mathbb I(\omega,\theta):=\int_{\mathbb R^3}\int_{\mathbb R^3}\rho(y)\rho(z)e^{{\rm i}c_1\omega \theta\cdot(y+z)}e^{{\rm i}c_2\omega |y-z|}\mathbb K(y,z)dydz
\end{equation} 
and the components in the third part are linear combinations of 
\begin{equation}\label{e12}
\mathbb J(\omega,\theta):=\omega^{-2}\int_{\mathbb R^3}\int_{\mathbb R^3}\rho(y)\rho(z)e^{{\rm i}c_1\omega \theta\cdot(y+z)}\beta(y,z,\omega)\mathbb K(y,z)dydz,
\end{equation}
where $c_1,c_2\in\{c_{\rm s}, c_{\rm p}\}$ and 
\begin{eqnarray*}\label{e13}
\mathbb K(y,z)=\frac{(y_1-z_1)^{p_1}(y_2-z_2)^{p_2}(y_3-z_3)^{p_3}}{|y-z|^{p_4}}.
\end{eqnarray*}
Here, $(p_1,p_2,p_3,p_4)\in S_{\mathbb I}$ for $\mathbb I(\omega,\theta)$ and $(p_1,p_2,p_3,p_4)\in S_{\mathbb J}$ for $\mathbb J(\omega,\theta)$ with
\begin{align*}
S_{\mathbb I}&:=\{(0,0,0,1),(2,0,0,3),(0,2,0,3),(0,0,2,3),(1,1,0,3),(1,0,1,3),(0,1,1,3)\},\\
S_{\mathbb J}&:=\{(0,0,0,3),(2,0,0,5),(0,2,0,5),(0,0,2,5),(1,1,0,5),(1,0,1,5),(0,1,1,5)\},
\end{align*}
such that $p_1+p_2+p_3-p_4=-1$ for $S_{\mathbb I}$ and  $p_1+p_2+p_3-p_4=-3$ for $S_{\mathbb J}$. 

Since the components of $\boldsymbol{u}_{2,\rm p}^{\infty}$ and $\boldsymbol{u}_{2,\rm s}^{\infty}$ are linear combinations of $\mathbb I(\omega, d)$ and $\mathbb J(\omega,\theta)$, Theorem \ref{thm4} can be obtained directly from the following lemma, whose proof is technical and put in the appendix for the purpose of readability. 

\begin{lemma}\label{lemma4}
Let the random potential $\rho$ satisfy Assumption \ref{as:rho}, and $\mathbb I(\omega,\theta)$ and $\mathbb J(\omega,\theta)$ be given by (\ref{e11}) and (\ref{e12}), respectively. For all $\theta\in\mathbb S^2$, it holds almost surely that 
\begin{align}\label{e16}
\lim_{Q\to\infty}\frac{1}{Q}\int_Q^{2Q}\omega^{m}|\mathbb I(\omega,\theta)|^2d\omega &=0,\\\label{e17}
\lim_{Q\to\infty}\frac{1}{Q}\int_Q^{2Q}\omega^{m}|\mathbb J(\omega,\theta)|^2d\omega &=0.
\end{align}
\end{lemma}

\subsection{The higher order far-field patterns}\label{sec:3.3}

It follows from (\ref{c7}) that the higher order backscattered far-field patterns can be expressed by
\begin{align}\label{f1}
{\boldsymbol b}^{\infty}_{\rm p}(- \theta,\omega, \theta)&=\sum_{j=3}^{\infty}{\boldsymbol u}^{\infty}_{j,\rm p}(- \theta,\omega, \theta),\\\label{f2}
{\boldsymbol b}^{\infty}_{\rm s}(- \theta,\omega, \theta)&=\sum_{j=3}^{\infty}{\boldsymbol u}^{\infty}_{j,\rm s}(- \theta,\omega, \theta),
\end{align}
where
\begin{align}\label{f3}
{\boldsymbol u}^{\infty}_{j,\rm p}(- \theta,\omega, \theta)&= -\frac{c_{\rm p}^2}{4\pi} \theta\otimes  \theta\int_{\mathbb R^3}e^{{\rm i}c_{\rm p}\omega  \theta\cdot z}\rho(z){\boldsymbol u}_{j-1}(z)dz,\\\label{f4}
{\boldsymbol u}^{\infty}_{j,\rm s}(- \theta,\omega, \theta)&= -\frac{c_{\rm s}^2}{4\pi}({\boldsymbol I}- \theta\otimes  \theta)\int_{\mathbb R^3}e^{{\rm i}c_{\rm s}\omega  \theta\cdot z}\rho(z){\boldsymbol u}_{j-1}(z)dz.
\end{align}
The goal is to estimate the order of ${\boldsymbol b}^{\infty}_{\rm p}$ and ${\boldsymbol b}^{\infty}_{\rm s}$ with respect to the frequency $\omega$, and show that the contribution of the higher order far-field patterns can be ignored as well.

\begin{theorem}\label{tm:6}
For any $s\in(\frac{3-m}2,\frac12)$, it holds almost surely that 
\begin{align}\label{f5}
\sup_{\theta\in\mathbb S^2}\left|{\boldsymbol b}_{\rm p}^{\infty}(- \theta,\omega, \theta)\right|&\lesssim \omega^{-2+6s},\\\label{f6}
\sup_{\theta\in\mathbb S^2}\left|{\boldsymbol b}_{\rm s}^{\infty}(- \theta,\omega, \theta)\right|&\lesssim \omega^{-2+6s}.
\end{align}
\end{theorem}

\begin{proof}
Define a cutoff function $\chi\in C_0^{\infty}(\mathbb R^3)$ supported in bounded domain $U$ such that $D\subset U$ and $\chi(z)=1$ if $z\in D$. For any $s\in(\frac{3-m}2,\frac12)$, $p\geq 3/s$ and $p'$ satisfying $1/p+1/p'=1$, 
it follows from (\ref{f1}) and (\ref{f3}) that
\begin{align}\nonumber
\left|{\boldsymbol b}_{\rm p}^{\infty}(- \theta,\omega, \theta)\right|&\lesssim \bigg|\int_{\mathbb R^3}e^{{\rm i}c_{\rm p}\omega  \theta\cdot z}\chi(z)\rho(z)\sum_{j=3}^{\infty}{\boldsymbol u}_{j-1}(z)dz\bigg|\\\nonumber
&\lesssim \|\rho\|_{W^{-s,p}(\mathbb R^3)}\Big\|\chi e^{{\rm i}c_{\rm p}\omega \theta\cdot (\cdot)}\chi\sum_{j=3}^{\infty}{\boldsymbol u}_{j-1}\Big\|_{\boldsymbol W^{s,p'}(\mathbb R^3)}\\\nonumber
&\lesssim \|\rho\|_{W^{-s,p}(\mathbb R^3)}\big\|\chi e^{{\rm i}c_{\rm p}\omega \theta\cdot (\cdot)}\big\|_{H^{s}(\mathbb R^3)}\Big\|\chi\sum_{j=3}^{\infty}{\boldsymbol u}_{j-1}\Big\|_{\boldsymbol H^{s}(\mathbb R^3)}\\\label{f7}
&\lesssim\|\rho\|_{W^{-s,p}(\mathbb R^3)}\big\|\chi e^{{\rm i}c_{\rm p}\omega \theta\cdot (\cdot)}\big\|_{H^{s}(\mathbb R^3)}\Big\|\sum_{j=3}^{\infty}{\boldsymbol u}_{j-1}\Big\|_{\boldsymbol H^{s}_{-1}(\mathbb R^3)},
\end{align}
where we used the facts (cf. \cite{CHL,LLW})
\[
\|fg\|_{W^{s,p'}(\mathbb R^3)}\lesssim\|f\|_{H^s(\mathbb R^3)}\|g\|_{H^s(\mathbb R^3)}\quad\forall~f,g\in\mathcal S
\]
for $p\ge\frac3s$ and $p'$ satisfying $\frac1p+\frac1{p'}=1$, and 
\[
\|\chi\boldsymbol u\|_{\boldsymbol H^s(\mathbb R^3)}\lesssim\|\boldsymbol u\|_{\boldsymbol H^s_{-2}(\mathbb R^3)}\lesssim\|\boldsymbol u\|_{\boldsymbol H^s_{-1}(\mathbb R^3)}\quad\forall~ \boldsymbol u\in\boldsymbol{\mathcal S}
\] 
with $\boldsymbol{\mathcal S}$ being dense in $\boldsymbol H^s_{-1}(\mathbb R^3)$. It is easy to check that 
\begin{eqnarray*}\label{f8}
\big\|\chi e^{{\rm i}c_{\rm p}\omega \theta\cdot (\cdot)}\big\|_{L^2(\mathbb R^3)}\lesssim 1,\quad \big\|\chi e^{{\rm i}c_{\rm p}\omega \theta\cdot (\cdot)}\big\|_{H^1(\mathbb R^3)}\lesssim \omega.
\end{eqnarray*}
Using the interpolation between spaces $L^2(\mathbb R^3)$ and $H^1(\mathbb R^3)$ yields 
\begin{eqnarray}\label{f9}
\big\|\chi e^{{\rm i}c_{\rm p}\omega \theta\cdot(\cdot)}\big\|_{H^{s}(\mathbb R^3)}\lesssim \big\|\chi e^{{\rm i}c_{\rm p}\omega \theta\cdot(\cdot)}\big\|_{L^2(\mathbb R^3)}^{1-s}\big\|\chi e^{{\rm i}c_{\rm p}\omega \theta\cdot(\cdot)}\big\|^s_{H^1(\mathbb R^3)}\lesssim \omega^s.
\end{eqnarray}
Note also that
\begin{align}\nonumber
\Big\|\sum_{j=3}^{\infty}\boldsymbol{u}_{j-1}\Big\|_{\boldsymbol H^s_{-1}(\mathbb R^3)}&=\Big\|\sum_{j=2}^{\infty}\boldsymbol{u}_j\Big\|_{H^s_{-1}(\mathbb R^3)}
\lesssim\sum_{j=2}^{\infty}\|{\mathcal K}_{\omega}^j\boldsymbol{u}_0\|_{H^s_{-1}(\mathbb R^3)}\\\nonumber
&\lesssim \sum_{j=2}^{\infty}\|{\mathcal K}_{\omega}\|^{j-1}_{\mathcal{L}(H^s_{-1}(\mathbb R^3),H^s_{-1}(\mathbb R^3))}\|{\mathcal K}_{\omega}\|_{\mathcal{L}(H^s_{-1}(D),H^s_{-1}(\mathbb R^3))}\|\boldsymbol{u}_0\|_{H^s_{-1}(D)}\\\nonumber
&\lesssim \sum_{j=2}^{\infty}\|{\mathcal K}_{\omega}\|^j_{\mathcal{L}(H^s_{-1}(\mathbb R^3),H^s_{-1}(\mathbb R^3))}\|\boldsymbol{u}_0\|_{H^s(D)}\\\label{f10}
&\lesssim \sum_{j=2}^{\infty}\omega^{j(-1+2s)}\omega^s\lesssim\omega^{-2+5s},
\end{align}
where we use Lemma \ref{lemma1} and the inequality $\|\cdot\|_{H^s_{-1}(\mathbb R^3)}\leq \|\cdot\|_{H^s(\mathbb R^3)}$ which can be easily checked by the definition. Combining (\ref{f7})--(\ref{f10}), we get 
\begin{eqnarray*}\label{f11}
\sup_{\theta\in\mathbb S^2}\left|{\boldsymbol b}_{\rm p}^{\infty}(-\theta,\omega,\theta)\right|\lesssim \omega^{-2+6s},
\end{eqnarray*}
which completes the proof of (\ref{f5}). The estimate (\ref{f6}) can be obtained similarly by using (\ref{f2}) and (\ref{f4}).
\end{proof}

Now we are in the position to prove the main result of the work. 

\begin{proof}[Proof of Theorem \ref{thm1}]
Recall from \eqref{c7} that the compressional far-field pattern $\boldsymbol{u}_{\rm p}^{\infty}$ has the form
\begin{eqnarray*}\label{g4}
\boldsymbol{u}_{\rm p}^{\infty}(-\theta,\omega,\theta)=\boldsymbol{u}_{1,\rm p}^{\infty}(-\theta,\omega,\theta)+\boldsymbol{u}_{2,\rm p}^{\infty}(-\theta,\omega,\theta)+\boldsymbol{b}_{\rm p}^{\infty}(-\theta,\omega,\theta).
\end{eqnarray*}
A simple calculation gives 
\begin{eqnarray*}\label{g5}
\frac{1}{Q}\int_{Q}^{2Q}\omega^m\boldsymbol{u}_{\rm p}^{\infty}(-\theta,\omega,\theta)\cdot\overline{\boldsymbol{u}_{\rm p}^{\infty}(-\theta,\omega+\tau,\theta)}d\omega=\sum_{j=1}^9I_j,
\end{eqnarray*}
where 
\begin{align*}
I_1&=\frac{1}{Q}\int_{Q}^{2Q}\omega^m\boldsymbol{u}_{1,\rm p}^{\infty}(-\theta,\omega,\theta)\cdot\overline{\boldsymbol{u}_{1,\rm p}^{\infty}(-\theta,\omega+\tau,\theta)}d\omega,\\
I_2&=\frac{1}{Q}\int_{Q}^{2Q}\omega^m\boldsymbol{u}_{1,\rm p}^{\infty}(-\theta,\omega,\theta)\cdot\overline{\boldsymbol{u}_{2,\rm p}^{\infty}(-\theta,\omega+\tau,\theta)}d\omega,\\
I_3&=\frac{1}{Q}\int_{Q}^{2Q}\omega^m\boldsymbol{u}_{1,\rm p}^{\infty}(-\theta,\omega,\theta)\cdot\overline{\boldsymbol{b}_{\rm p}^{\infty}(-\theta,\omega+\tau,\theta)}d\omega,\\
I_4&=\frac{1}{Q}\int_{Q}^{2Q}\omega^m\boldsymbol{u}_{2,\rm p}^{\infty}(-\theta,\omega,\theta)\cdot\overline{\boldsymbol{u}_{1,\rm p}^{\infty}(-\theta,\omega+\tau,\theta)}d\omega,\\
I_5&=\frac{1}{Q}\int_{Q}^{2Q}\omega^m\boldsymbol{u}_{2,\rm p}^{\infty}(-\theta,\omega,\theta)\cdot\overline{\boldsymbol{u}_{2,\rm p}^{\infty}(-\theta,\omega+\tau,\theta)}d\omega,\\
I_6&=\frac{1}{Q}\int_{Q}^{2Q}\omega^m\boldsymbol{u}_{2,\rm p}^{\infty}(-\theta,\omega,\theta)\cdot\overline{\boldsymbol{b}_{\rm p}^{\infty}(-\theta,\omega+\tau,\theta)}d\omega,\\
I_7&=\frac{1}{Q}\int_{Q}^{2Q}\omega^m\boldsymbol{b}_{\rm p}^{\infty}(-\theta,\omega,\theta)\cdot\overline{\boldsymbol{u}_{1,\rm p}^{\infty}(-\theta,\omega+\tau,\theta)}d\omega,\\
I_8&=\frac{1}{Q}\int_{Q}^{2Q}\omega^m\boldsymbol{b}_{\rm p}^{\infty}(-\theta,\omega,\theta)\cdot\overline{\boldsymbol{u}_{2,\rm p}^{\infty}(-\theta,\omega+\tau,\theta)}d\omega,\\
I_9&=\frac{1}{Q}\int_{Q}^{2Q}\omega^m\boldsymbol{b}_{\rm p}^{\infty}(-\theta,\omega,\theta)\cdot\overline{\boldsymbol{b}_{\rm p}^{\infty}(-\theta,\omega+\tau,\theta)}d\omega.
\end{align*}
It follows from Theorems \ref{thm3}, \ref{thm4} and \ref{tm:6} that 
\[
\lim_{Q\to\infty}I_1=C_{\rm p}\hat{\phi}(2c_{\rm p}\tau  \theta),
\] 
\begin{eqnarray*}
\Big|\lim_{Q\to\infty}I_2\Big|\leq \lim_{Q\to\infty}\left[\frac{1}{Q}\int_{Q}^{2Q}\omega^m|\boldsymbol{u}_{1,\rm p}^{\infty}(-\theta,\omega,\theta)|^2d\omega\right]^{\frac{1}{2}}\left[\frac{1}{Q}\int_Q^{2Q}\omega^m|\boldsymbol{u}_{2,\rm p}^{\infty}(-\theta,\omega+\tau,\theta)|^2d\omega\right]^{\frac{1}{2}}=0,
\end{eqnarray*}
 and 
\begin{align*}
\Big|\lim_{Q\to\infty}I_3\Big|&\leq \lim_{Q\to\infty}\left[\frac{1}{Q}\int_{Q}^{2Q}\omega^m|\boldsymbol{u}_{1,\rm p}^{\infty}(-\theta,\omega,\theta)|^2d\omega\right]^{\frac{1}{2}}\left[\frac{1}{Q}\int_Q^{2Q}\omega^m|\boldsymbol{b}_{\rm p}^{\infty}(-\theta,\omega+\tau,\theta)|^2d\omega\right]^{\frac{1}{2}}\\
&\lesssim \lim_{Q\to\infty}\left[\frac{1}{Q}\int_{Q}^{2Q}\omega^m|\boldsymbol{u}_{1,\rm p}^{\infty}(-\theta,\omega,\theta)|^2d\omega\right]^{\frac{1}{2}}\left[\frac{1}{Q}\int_Q^{2Q}\omega^m(\omega+\tau)^{-4+12s}d\omega\right]^{\frac{1}{2}}\\
&\leq \lim_{Q\to\infty}\left[\frac{1}{Q}\int_{Q}^{2Q}\omega^m|\boldsymbol{u}_{1,\rm p}^{\infty}(-\theta,\omega,\theta)|^2d\omega\right]^{\frac{1}{2}}\left[\frac{1}{Q}\int_Q^{2Q}\omega^{m-4+12s}d\omega\right]^{\frac{1}{2}}\\
&= \left[C_{\rm p}\hat{\phi}(0)\right]^{\frac{1}{2}}\lim_{Q\to\infty}\left[\frac{(2Q)^{m-3+12s}-Q^{m-3+12s}}{(m-3+12s)Q}\right]^{\frac{1}{2}}=0
\end{align*}
for any $s\in(\frac{3-m}2,\frac13-\frac m{12})$, where the domain is nonempty since $m>\frac{14}5$ and hence such an $s$ exists.

Based on the same procedure used for the estimates of $I_2$ and $I_3$, we may also show that $I_4$, $I_5$, $I_6$, $I_7$, $I_8$, $I_9$ have limit zeros when $Q\to\infty$, and conclude that (\ref{a5}) holds. The result (\ref{a6}) for the shear far-field pattern can be obtained similarly. The details are omitted. 

Due to continuation of a dense set and the fact that $\phi$ is analytic, the microlocal strength $\phi$ can be uniquely determined by $\{\hat{\phi}(2c_{\rm p}\tau  \theta)\}_{(\tau,\theta)\in\Theta}$ or $\{\hat{\phi}(2c_{\rm s}\tau  \theta)\}_{(\tau,\theta)\in\Theta}$ with $\Theta$ being any open domain of $\mathbb R_+\times\mathbb S^2$. 
\end{proof}

\section{Conclusion}

In this paper, we have studied the inverse scattering problem for the three-dimensional time-harmonic elastic wave equation with a random potential. The random potential is assumed to be a microlocally isotropic Gaussian random field such that its covariance operator is a classical pseudo-differential operator, and should be interpreted as a distribution. For the direct problem, we prove that it is well-posed in the distribution sense by examining the equivalent Lippmann--Schwinger integral equation. For the inverse scattering problem, we show that the strength of the random potential can be uniquely determined by a single realization of the high frequency limit of the averaged compressional (resp. shear) backscattered far-field pattern of the scattered wave associated to the compressional (resp. shear) plane incident wave.

This paper is concerned with the three-dimensional problem in a homogeneous medium, where the Green tensor has an explicit form which makes it possible to get the reconstruction formula of the strength. The inverse scattering problem is open for the two-dimensional case, where the Green tensor consists of the Hankel functions. Although it is a  two-dimensional problem, the difficult arises from the complex form of the Green tensor and more sophisticated analysis is needed for the involved Hankel functions. The problems are even more challenging if the medium is inhomogeneous where the explicit Green tensors are not available any longer. We hope to be able to report the progress on these problems elsewhere in the future. 

\appendix

\section{Appendix}\label{app:A}

This appendix is devoted to the proof of Lemma \ref{lemma4}. First, we present some preliminary results about the microlocally isotropic Gaussian random fields in $\mathbb R^n$. The following lemma concerns the kernel function of the covariance operator for the Gaussian random field and the details can be found in \cite{CHL}.

\begin{lemma}\label{lemmaadd4}
Let $\rho$ be a microlocally isotropic Gaussian random field of order
$-m$ in $D\subset\mathbb R^n$ with $m\in(n-1,n]$. Then there exists two functions $F_0\in C^{\infty}(\mathbb R^n\times\mathbb R^n)$ and $F_{\alpha}\in C^{0,\alpha}(\mathbb R^n\times\mathbb R^n)$ with $\alpha\in(0,m-(n-1))$  supported on $D\times D$ such that the covariance kernel $K_\rho$ has the form
\begin{eqnarray*}
K_\rho(x, y)=\begin{cases}
           F_0(x,y)|x-y|^{-(n-m)}+F_{\alpha}(x,y)&\quad{\rm if}\quad m\in(n-1,n),\\
           F_0(x,y){\rm log}|x-y|+F_{\alpha}(x,y)&\quad{\rm if}\quad m=n.
          \end{cases}
\end{eqnarray*}
\end{lemma}

Define the smooth potential function 
\begin{equation}\label{rhoepsilon}
\rho_{\varepsilon}:=\rho\ast\varphi_{\varepsilon},
\end{equation}
which is obtained from taking the convolution of the rough potential $\rho$ with the polishing function $\varphi_{\varepsilon}(x):=\varepsilon^{-n}\varphi(x/\varepsilon)$, where $\varphi\in C_0^{\infty}(\mathbb R^n)$ satisfies $\int_{\mathbb R^n}\varphi(x)dx=1$ and ${\rm supp}\varphi\subset\{x\in\mathbb R^n: |x|\leq 1/4\}$. Denote the covariance kernel of $\rho_{\varepsilon}$ by 
\[
K_{\rho_{\varepsilon}}(x,y)=\mathbb E[\rho_{\varepsilon}(x)\rho_{\varepsilon}(y)]. 
\]

\begin{lemma}\label{lemmaadd5}
Let $\rho$ satisfy assumptions in Lemma \ref{lemmaadd4} and $\rho_{\varepsilon}$ be given in \eqref{rhoepsilon}. Then the covariance kernel $K_{\rho_{\varepsilon}}(x,y)$ converges to $K_\rho(x,y)$ point-wisely as $\varepsilon\to0$.
\end{lemma}

\begin{proof}
By the definition of $K_{\rho_\varepsilon}$, we have 
\begin{align}\nonumber
K_{\rho_{\varepsilon}}(x,y)-K_\rho(x,y)
&=\int_{\mathbb R^n}\int_{\mathbb R^n}\left(\mathbb E\left[\rho(x-z_1)\rho(y-z_2)\right]-K_{\rho}(x,y)\right)\varphi_{\varepsilon}(z_1)\varphi_{\varepsilon}(z_2)dz_1dz_2\\\label{y1}
&=\int_{\mathbb R^n}\int_{\mathbb R^n}\left[K_\rho(x-z_1,y-z_2)-K_\rho(x,y)\right]\varphi_{\varepsilon}(z_1)\varphi_{\varepsilon}(z_2)dz_1dz_2.
\end{align}

For the case $m=n$, we get from Lemma \ref{lemmaadd4} that 
\begin{align*}
&K_\rho(x-z_1,y-z_2)-K_\rho(x,y)\\
&=F_0(x-z_1,y-z_2){\rm log}|(x-z_1)-(y-z_2)|+F_{\alpha}(x-z_1,y-z_2)-F_0(x,y){\rm log}|x-y|-F_{\alpha}(x,y)\\
&=F_0(x-z_1,y-z_2)\left[{\rm log}|(x-z_1)-(y-z_2)|-{\rm log}|x-y|\right]\\
&\quad +\left[F_0(x-z_1,y-z_2)-F_0(x,y)\right]{\rm log}|x-y|+F_{\alpha}(x-z_1,y-z_2)-F_{\alpha}(x,y).
\end{align*}
Note that in (\ref{y1}), for any $z_1, z_2\in{\rm supp}\varphi_{\varepsilon}$, it holds $|z_1|\leq \varepsilon/4$ and $|z_2|\leq \varepsilon/4$. 
For any $\delta>0$ and any fixed $x,y\in D$ with $x\neq y$, there exists a sufficiently small $\varepsilon=\varepsilon(\delta,x,y)>0$ such that if $z_1,z_2\in{\rm supp}\varphi_{\varepsilon}$, then 
\begin{eqnarray*}\label{y3}
&&\left|F_0(x-z_1,y-z_2)-F_0(x,y)\right|< \delta,\quad \left|F_{\alpha}(x-z_1,y-z_2)-F_{\alpha}(x,y)\right|<\delta,\\\label{y4}
&&\left|{\rm log}|(x-z_1)-(y-z_2)|-{\rm log}|x-y|\right|=\left|{\rm log}\left|1-\frac{z_1-z_2}{x-y}\right|\right|<\delta
\end{eqnarray*}
due to the continuity of $F_0$ and $F_\alpha$,  and $F_0(x-z_1,y-z_2)$ is bounded.
Hence, for any $\delta>0$ and fixed $x,y\in D$ with $x\neq y$, there exists $\varepsilon=\varepsilon(\delta,x,y)>0$ such that 
\begin{eqnarray*}\label{y5}
\left|K_\rho(x-z_1,y-z_2)-K_\rho(x,y)\right|<\delta\quad\forall~z_1,z_2\in{\rm supp}\varphi_{\varepsilon},
\end{eqnarray*}
and (\ref{y1}) satisfies
\begin{eqnarray*}\label{y6}
\left|K_{\rho_{\varepsilon}}(x,y)-K_\rho(x,y)\right|<\delta.
\end{eqnarray*}

For the case $m\in(n-1,n)$, note that for any $\delta>0$ and fixed $x,y\in D$ with $x\neq y$, there exists $\varepsilon=\varepsilon(\delta,x,y)>0$ such that if $z_1,z_2\in{\rm supp}\varphi_{\varepsilon}$, then
\begin{eqnarray*}
\left|\left|(x-z_1)-(y-z_2)\right|^{-(n-m)}-|x-y|^{-(n-m)}\right|
=|x-y|^{-(n-m)}\left|\left|1-\frac{z_1-z_2}{x-y}\right|^{-(n-m)}-1\right|<\delta.
\end{eqnarray*}
A similar argument as the case $m=n$ leads to the point-wise convergence of $K_{\rho_{\varepsilon}}(x,y)$.
\end{proof}

\begin{lemma}\label{lemmaadd7}
For any given $x,y\in\mathbb R^n$ with $x\neq y$ and any $\varepsilon<\frac12|x-y|$, the covariance kernel $K_{\rho_{\varepsilon}}$ of $\rho_\varepsilon$ admits the following estimates:
\begin{align*}
\left|K_{\rho_{\varepsilon}}(x,y)\right|&\lesssim \left|{\rm log}|x-y|\right|+O(1)\qquad{\rm for}\;\; m=n,\\\label{y10}
\left|K_{\rho_{\varepsilon}}(x,y)\right|&\lesssim |x-y|^{-(3-m)}+O(1)\quad{\rm for}\;\;m\in (n-1, n).
\end{align*}
\end{lemma}

\begin{proof}
It follows from (\ref{y1}) that 
\begin{eqnarray}\label{y11}
K_{\rho_{\varepsilon}}(x,y)=\int_{\mathbb R^n}\int_{\mathbb R^n}K_{\rho}(x-z_1,y-z_2)\varphi_{\varepsilon}(z_1)\varphi_{\varepsilon}(z_2)dz_1dz_2.
\end{eqnarray}
For $z_1,z_2\in{\rm supp}\varphi_\varepsilon$ with $\epsilon<\frac12|x-y|$, we obtain
\begin{eqnarray*}\label{y12}
|z_1-z_2|\leq |z_1|+|z_2|\leq\frac14\varepsilon+\frac14\varepsilon=\frac{1}{2}\varepsilon<\frac{1}{4}|x-y|,
\end{eqnarray*}
which gives the following equivalence between $|x-y|$ and $|(x-z_1)-(y-z_2)|$:
\begin{eqnarray*}\label{y13}
\frac{3}{4}|x-y|<|x-y|-|z_1-z_2|\leq |(x-z_1)-(y-z_2)|\leq |x-y|+|z_1-z_2|<\frac{5}{4}|x-y|.
\end{eqnarray*}
Hence,
\begin{eqnarray*}\label{y14}
&&\left|{\rm log}|(x-z_1)-(y-z_2)|\right|\lesssim \left|{\rm log}|x-y|\right|,\\\label{y15}
&&\left|(x-z_1)-(y-z_2)\right|^{-(3-m)}\lesssim |x-y|^{-(3-m)}.
\end{eqnarray*}
Based on Lemma \ref{lemmaadd4} and the fact that $F_0$ and $F_\alpha$ are continuous and supported on $D\times D$, we get 
\begin{align*}
|K_{\rho}(x-z_1,y-z_2)|&=|F_0(x-z_1,y-z_2){\rm log}|(x-z_1)-(y-z_2)|+F_{\alpha}(x-z_1,y-z_2)|\notag\\
&\lesssim \left|{\rm log}|x-y|\right|+O(1)
\end{align*}
for $m=n$, and 
\begin{align*}
|K_{\rho}(x-z_1,y-z_2)|&=|F_0(x-z_1,y-z_2)|(x-z_1)--(y-z_2)|^{-(3-m)}+F_{\alpha}(x-z_1,y-z_2)|\notag\\
&\lesssim |x-y|^{-(3-m)}+O(1)
\end{align*}
for $m\in (n-1,n)$, which, together with \eqref{y11}, completes the proof.
\end{proof}

\begin{proof}[Proof of Lemma \ref{lemma4}]
For simplicity, let $\mathbb F(\omega,\theta)$ denote $\mathbb I(\omega,\theta)$ or $\mathbb J(\omega,\theta)$. Note that 
\begin{eqnarray*}\label{e18}
\frac{1}{Q}\int_Q^{2Q}\omega^m|\mathbb F(\omega,\theta)|^2d\omega=\frac{2}{2Q}\int_1^{2Q}\omega^m|\mathbb F(\omega,\theta)|^2d\omega-\frac{1}{Q}\int_1^{Q}\omega^m|\mathbb F(\omega,\theta)|^2d\omega,
\end{eqnarray*}
where
\begin{eqnarray*}\label{e20}
\frac{1}{Q}\int_1^{Q}\omega^m|\mathbb F(\omega,\theta)|^2d\omega=\int_1^{\infty}\frac{\omega {\bf 1}_{[1,Q]}(\omega)}{Q}\omega^{m-1}|\mathbb F(\omega,\theta)|^2d\omega
\end{eqnarray*}
with ${\bf 1}_{[1,Q]}(\omega)$ standing for the characteristic function of the interval $[1,Q]$, 
\[
\left|\frac{\omega {\bf 1}_{[1,Q]}(\omega)}{Q}\omega^{m-1}|\mathbb F(\omega,\theta)|^2\right|\lesssim \omega^{m-1}|\mathbb F(\omega,\theta)|^2
\]
and
\begin{eqnarray*}\label{e21}
\lim_{Q\to\infty}\frac{\omega {\bf 1}_{[1,Q]}(\omega)}{Q}=0
\end{eqnarray*}
point-wisely. 
Note also that, if 
\begin{eqnarray}\label{e22}
\int_1^{\infty}\omega^{m-1}|\mathbb F(\omega,\theta)|^2d\omega<\infty\quad{a.s.},
\end{eqnarray}
then the dominate convergence theorem leads to
\begin{eqnarray*}\label{e19}
\lim_{Q\to\infty}\frac{1}{Q-1}\int_1^{Q}\omega^m|\mathbb F(\omega,\theta)|^2d\omega=0\quad a.s.,
\end{eqnarray*}
which yields (\ref{e16}) and (\ref{e17}) in Lemma \ref{lemma4}.
Since $\mathbb F(\omega, \theta)$ is continuous with respect to $\theta\in\mathbb S^2$ according to (\ref{e11}) and (\ref{e12}), to prove (\ref{e22}), it then suffices to show that
\begin{eqnarray}\label{e24}
\mathbb E\left[\int_{\mathbb S^2}\int_1^{\infty}\omega^{m-1}|\mathbb F(\omega,\theta)|^2d\omega dS(\theta)\right]<\infty.
\end{eqnarray}

According to  Assumption \ref{as:rho} and Lemma \ref{lemmaadd3}, the random potential $\rho$ involved in $\mathbb F(\omega,\theta)$ is too rough to exist point-wisely and should be interpreted as a distribution. To deal with the roughness, next, we consider the smoothed potential function $\rho_{\varepsilon}$ given in (\ref{rhoepsilon}). Define
\begin{align*}
\mathbb I_{\varepsilon}(\omega,\theta)&:=\int_{\mathbb R^3}\int_{\mathbb R^3}\rho_{\varepsilon}(y)\rho_{\varepsilon}(z)e^{{\rm i}c_1\omega \theta\cdot(y+z)}e^{{\rm i}c_2\omega |y-z|}\mathbb K(y,z)dydz,\\
\mathbb J_{\varepsilon}(\omega,\theta)&:=\omega^{-2}\int_{\mathbb R^3}\int_{\mathbb R^3}\rho_{\varepsilon}(y)\rho_{\varepsilon}(z)e^{{\rm i}c_1\omega  \theta\cdot(y+z)}\beta(y,z,\omega)\mathbb K(y,z)dydz
\end{align*}
by replacing $\rho$ with $\rho_{\varepsilon}$ in (\ref{e11}) and (\ref{e12}). By the Wick formula and Lemma \ref{lemmaadd5}, we obtain 
\begin{align*}
&\lim_{\varepsilon\to0}\mathbb E\left[\rho_{\varepsilon}(y)\rho_{\varepsilon}(z)\rho_{\varepsilon}(u)\rho_{\varepsilon}(v)\right]\\
&=\lim_{\varepsilon\to0}\big(\mathbb E\left[\rho_{\varepsilon}(y)\rho_{\varepsilon}(z)\right]\mathbb E\left[\rho_{\varepsilon}(u)\rho_{\varepsilon}(v)\right]
+\mathbb E\left[\rho_{\varepsilon}(y)\rho_{\varepsilon}(u)\right]\mathbb E\left[\rho_{\varepsilon}(z)\rho_{\varepsilon}(v)\right]\\
&\qquad +\mathbb E\left[\rho_{\varepsilon}(y)\rho_{\varepsilon}(v)\right]\mathbb E\left[\rho_{\varepsilon}(u)\rho_{\varepsilon}(z)\right]\big)\\
&=K_\rho(y,z)K_\rho(u,v)+K_\rho(y,u)K_\rho(z,v)+K_\rho(y,v)K_\rho(u,z)
\end{align*}
point-wisely. Hence, we have 
\begin{align*}
\lim_{\varepsilon\to 0}\mathbb E\left|\mathbb I_{\varepsilon}(\omega,\theta)\right|^2
&=\lim_{\varepsilon\to 0}\int_{\mathbb R^3}\int_{\mathbb R^3}\int_{\mathbb R^3}\int_{\mathbb R^3}\mathbb E\left[\rho_{\varepsilon}(y)\rho_{\varepsilon}(z)\rho_{\varepsilon}(u)\rho_{\varepsilon}(v)\right]
 e^{{\rm i}c_1\omega  \theta\cdot(y+z)}e^{{\rm i}c_2\omega |y-z|}\\
 &\quad \times e^{-{\rm i}c_1\omega  \theta\cdot(u+v)}e^{-{\rm i}c_2\omega |u-v|}\mathbb K(y,z)\mathbb K(u,v)dydzdudv\\
 &=\int_{\mathbb R^3}\int_{\mathbb R^3}\int_{\mathbb R^3}\int_{\mathbb R^3}\left[K_\rho(y,z)K_\rho(u,v)+K_\rho(y,u)K_\rho(z,v)+K_\rho(y,v)K_\rho(u,z)\right]\\
&\quad \times e^{{\rm i}c_1\omega  \theta\cdot(y+z)}e^{{\rm i}c_2\omega |y-z|}
 e^{-{\rm i}c_1\omega  \theta\cdot(u+v)}e^{-{\rm i}c_2\omega |u-v|}\mathbb K(y,z)\mathbb K(u,v)dydzdudv\\
&=\int_{\mathbb R^3}\int_{\mathbb R^3}\int_{\mathbb R^3}\int_{\mathbb R^3}\mathbb E\left[\rho(y)\rho(z)\rho(u)\rho(v)\right]
e^{{\rm i}c_1\omega  \theta\cdot(y+z)}e^{{\rm i}c_2\omega |y-z|}\\
&\quad \times e^{-{\rm i}c_1\omega  \theta\cdot(u+v)}e^{-{\rm i}c_2\omega |u-v|}\mathbb K(y,z)\mathbb K(u,v)dydzdudv\\
&=\mathbb E\left|\mathbb I(\omega,\theta)\right|^2.
\end{align*}
Similarly, we have
\begin{eqnarray*}
\lim_{\varepsilon\to 0}\mathbb E\left|\mathbb J_{\varepsilon}(\omega,\theta)\right|^2=\mathbb E\left|\mathbb J(\omega,\theta)\right|^2.
\end{eqnarray*}
Let $\mathbb F_{\varepsilon}$ denote $\mathbb I_{\varepsilon}$ or $\mathbb J_{\varepsilon}$. Then Fatou's lemma gives 
\begin{align*}
\mathbb E\int_{\mathbb S^2}\int_1^{\infty}\omega^{m-1}\left|\mathbb F(\omega,\theta)\right|^2d\omega dS(\theta)&\leq \varliminf_{\varepsilon\to 0}\mathbb E\int_{\mathbb S^2}\int_1^{\infty}\omega^{m-1}\left|\mathbb F_{\varepsilon}(\omega, \theta)\right|^2d\omega dS(\theta)\\
&\leq \varlimsup_{\varepsilon\to 0}\mathbb E\int_{\mathbb S^2}\int_1^{\infty}\omega^{m-1}\left|\mathbb F_{\varepsilon}(\omega, \theta)\right|^2d\omega dS(\theta).
\end{align*}
Hence, to prove (\ref{e24}), it is enough to show 
\begin{eqnarray}\label{e29}
\varlimsup_{\varepsilon\to 0}\mathbb E\int_{\mathbb S^2}\int_1^{\infty}\omega^{m-1}\left|\mathbb F_{\varepsilon}(\omega, \theta)\right|^2d\omega dS(\theta)<\infty.
\end{eqnarray}

To prove (\ref{e29}), the main idea is to rewrite $\mathbb F_{\varepsilon}$ as the Fourier transform of a proper function.  To this end, we define $y+z=g$ and $y-z=h=r\xi$ with $r=|h|$ and $\xi\in\mathbb S^2$.

(i) For the case $\mathbb F_{\varepsilon}=\mathbb I_{\varepsilon}$, it holds
\begin{align*}
\mathbb I_{\varepsilon}(\omega,\theta)&=\frac{1}{8}\int_{\mathbb R^3}\int_{\mathbb R^3}\rho_{\varepsilon}\left(\frac{g+h}{2}\right)\rho_{\varepsilon}\left(\frac{g-h}{2}\right)e^{{\rm i}c_1\omega  \theta\cdot g}e^{{\rm i}c_2\omega |h|}\mathbb K\left(\frac{g+h}{2}, \frac{g-h}{2}\right)dg dh\\
&=\frac{1}{8}\int_{\mathbb R^3}\int_0^{\infty}\int_{\mathbb S^2}\rho_{\varepsilon}\left(\frac{g+r\xi}{2}\right)\rho_{\varepsilon}\left(\frac{g-r\xi}{2}\right)e^{{\rm i}\omega (c_2r+c_1 \theta\cdot g)}r^{-1}\xi_1^{p_1}\xi_2^{p_2}\xi_3^{p_3}r^{2}dg drdS(\xi)\\
&=\int_{\mathbb R^3}\int_{\mathbb R}e^{{\rm i}\omega(c_2r+c_1 \theta\cdot g)}f_{\varepsilon}(g,r)dg dr\\
&=\int_{\mathbb R^3}e^{{\rm i}\omega c_1  \theta\cdot g}Tf_{\varepsilon}(g,\omega)dg,
\end{align*}
where we use the fact $p_1+p_2+p_3-p_4=-1$ for $S_{\mathbb I}$, and $f_{\varepsilon}$ and $Tf_{\varepsilon}$ are defined by
\begin{align*}
f_{\varepsilon}(g,r)&:=\frac{1}{8}r {\bf 1}_{[0,\infty)}(r)\int_{\mathbb S^2}\rho_{\varepsilon}\left(\frac{g+r\xi}{2}\right)\rho_{\varepsilon}\left(\frac{g-r\xi}{2}\right)\xi_1^{p_1}\xi_2^{p_2}\xi_3^{p_3}dS(\xi),\\
Tf_{\varepsilon}(g,\omega)&:=\int_{\mathbb R}e^{{\rm i}\omega c_2r}f_{\varepsilon}(g,r)dr.
\end{align*}
Since 
\begin{eqnarray*}
\frac{g+r\xi}{2},\quad\frac{g-r\xi}{2}\in \bigcup_{\varepsilon\in (0,1]}{\rm supp}\rho_{\varepsilon},
\end{eqnarray*}
there exists a constant $M$ depending on the support of $\rho$ such that 
\begin{eqnarray}\label{e33}
|g|+|r|=\left|\frac{g+r\xi}{2}+\frac{g-r\xi}{2}\right|+\left|\frac{g+r\xi}{2}-\frac{g-r\xi}{2}\right|\leq |g+r\xi|+|g-r\xi|\leq M.
\end{eqnarray}
An analogous argument as that used in \cite[Lemma 4.7]{CHL} gives
\begin{eqnarray}\label{eq:Tf}
\int_{\mathbb S^2}\omega^{m-1}\left|\int_{\mathbb R^3}e^{{\rm i}\omega c_1\theta\cdot g}Tf_{\varepsilon}(g,\omega)dg\right|^2dS(\theta)\lesssim \omega^{m-3}\int_{\mathbb R^3}\left|Tf_{\varepsilon}(g,\omega)\right|^2dg.
\end{eqnarray}
We then conclude that 
\begin{align}\nonumber
\int_{\mathbb S^2}\omega^{m-1}\left|\mathbb I_{\varepsilon}(\omega,\theta)\right|^2dS(\theta)&=\int_{\mathbb S^2}\omega^{m-1}\left|\int_{\mathbb R^3}e^{{\rm i}\omega c_1 \theta\cdot g}Tf_{\varepsilon}(g,\omega)dg\right|^2dS(\theta)\\\label{e34}
&\lesssim \omega^{m-3}\int_{\mathbb R^3}\left|Tf_{\varepsilon}(g,\omega)\right|^2dg.
\end{align}
Noting 
\begin{eqnarray}\label{e35}
Tf_{\varepsilon}(g,\omega)=\int_{\mathbb R}e^{{\rm i}\omega c_2r}f_{\varepsilon}(g,r)dr=\frac{1}{c_2}\int_{\mathbb R}e^{{\rm i}\omega \tilde{r}}f_{\varepsilon}(g, \tilde{r}/c_2)d\tilde{r}
\end{eqnarray}
with $\tilde{r}=c_2r$, we deduce from (\ref{e34}) and (\ref{e35}) that 
\begin{align}\nonumber
&\int_1^{\infty}\int_{\mathbb S^2}\omega^{m-1}\left|\mathbb I_{\varepsilon}(\omega,\theta)\right|^2dS(\theta)d\omega\lesssim \int_1^{\infty}\omega^{m-3}\int_{\mathbb R^3}\left|Tf_{\varepsilon}(g,\omega)\right|^2dgd\omega\\\nonumber
&=\int_{\mathbb R^3}\int_1^{\infty}\omega^{m-3}\int_{\mathbb R^3}\left|Tf_{\varepsilon}(g,\omega)\right|^2d\omega dg\lesssim \int_{\mathbb R^3}\int_1^{\infty}\left|Tf_{\varepsilon}(g,\omega)\right|^2d\omega dg\\\label{e36}
&\lesssim \int_{\mathbb R^3}\int_{\mathbb R}\left|f_{\varepsilon}(g,\tilde{r}/c_2)\right|^2d\tilde{r}dg,
\end{align}
where we use the Plancherel identity and the fact $\omega^{m-3}\leq 1$ for $\omega\geq 1$ and $m\in (14/5, 3]$. Taking the limit of the expectation of (\ref{e36}) yields
\begin{eqnarray*}\label{e37}
\varlimsup_{\varepsilon\to 0}\mathbb E\int_1^{\infty}\int_{\mathbb S^2}\omega^{m-1}\left|\mathbb I_{\varepsilon}(\omega,\theta)\right|^2dS(\theta)d\omega\lesssim \lim_{\varepsilon\to 0} \int_{|g|<M}\int_{|\tilde{r}|<c_2M} \mathbb E\left|f_{\varepsilon}(g,\tilde{r}/c_2)\right|^2d\tilde{r}dg,
\end{eqnarray*}
where we use (\ref{e33}). According to the definition of $f_{\varepsilon}$, we have
\begin{eqnarray*}\label{e38}
f_{\varepsilon}(g,\tilde{r}/c_2)=\frac{1}{8c_2}\tilde{r} {\bf 1}_{[0,\infty)}(\tilde{r})\int_{\mathbb S^2}\rho_{\varepsilon}\Big(\frac{g+(\tilde{r}\xi)/c_2}{2}\Big)\rho_{\varepsilon}\Big(\frac{g-(\tilde{r}\xi)/c_2}{2}\Big)\xi_1^{p_1}\xi_2^{p_2}\xi_3^{p_3}dS(\xi),
\end{eqnarray*}
which leads to 
\begin{eqnarray*}\label{e39}
\mathbb E\left|f_{\varepsilon}\left(g,\tilde{r}/c_2\right)\right|^2=\frac{1}{64c_2^2}\tilde{r}^2{\bf 1}_{[0,\infty)}(\tilde{r})\int_{\mathbb S^2}\int_{\mathbb S^2}A_{\varepsilon}(g,\tilde{r},\xi,\zeta)\xi_1^{p_1}\xi_2^{p_2}\xi_3^{p_3}\zeta_1^{p_1}\zeta_2^{p_2}\zeta_3^{p_3}dS(\xi)dS(\zeta)
\end{eqnarray*}
with
\begin{eqnarray*}\label{e40}
A_{\varepsilon}(g,\tilde{r},\xi,\zeta):=\mathbb E \bigg[\rho_{\varepsilon}\Big(\frac{g+(\tilde{r}\xi)/c_2}{2}\Big)\rho_{\varepsilon}\Big(\frac{g-(\tilde{r}\xi)/c_2}{2}\Big)
\rho_{\varepsilon}\Big(\frac{g+(\tilde{r}\zeta)/c_2}{2}\Big)\rho_{\varepsilon}\Big(\frac{g-(\tilde{r}\zeta)/c_2}{2}\Big)\bigg].
\end{eqnarray*}
 It follows from the Wick formula that 
 \begin{align*}
&A_{\varepsilon}(g,\tilde{r},\xi,\zeta)\\
&= \mathbb E \bigg[\rho_{\varepsilon}\Big(\frac{g+(\tilde{r}\xi)/c_2}{2}\Big)\rho_{\varepsilon}\Big(\frac{g-(\tilde{r}\xi)/c_2}{2}\Big)\bigg]\mathbb E\bigg[ \rho_{\varepsilon}\Big(\frac{g+(\tilde{r}\zeta)/c_2}{2}\Big)\rho_{\varepsilon}\Big(\frac{g-(\tilde{r}\zeta)/c_2}{2}\Big)\bigg]\\
&\quad  + \mathbb E \bigg[\rho_{\varepsilon}\Big(\frac{g+(\tilde{r}\xi)/c_2}{2}\Big)\rho_{\varepsilon}\Big(\frac{g+(\tilde{r}\zeta)/c_2}{2}\Big)\bigg]\mathbb E\bigg[ \rho_{\varepsilon}\Big(\frac{g-(\tilde{r}\xi)/c_2}{2}\Big)\rho_{\varepsilon}\Big(\frac{g-(\tilde{r}\zeta)/c_2}{2}\Big)\bigg]\\
&\quad \times\mathbb E \bigg[\rho_{\varepsilon}\Big(\frac{g+(\tilde{r}\xi)/c_2}{2}\Big)\rho_{\varepsilon}\Big(\frac{g-(\tilde{r}\zeta)/c_2}{2}\Big)\bigg]\mathbb E\bigg[ \rho_{\varepsilon}\Big(\frac{g+(\tilde{r}\zeta)/c_2}{2}\Big)\rho_{\varepsilon}\Big(\frac{g-(\tilde{r}\xi)/c_2}{2}\Big)\bigg].
\end{align*}
Applying Lemma \ref{lemmaadd7} implies that 
\begin{align}\nonumber
&\left|A_{\varepsilon}(g,\tilde{r},\xi,\zeta)\right|\\\nonumber
&\leq \left[{\rm log}\left(\tilde{r}/c_2\right)+O(1)\right]^2+\bigg[{\rm log}\Big|\frac{\tilde{r}}{2c_2}(\xi-\zeta)\Big|+O(1)\bigg]^2+\left[{\rm log}\Big|\frac{\tilde{r}}{2c_2}(\xi+\zeta)\Big|+O(1)\right]^2\\\nonumber
&\lesssim {\rm log}^2(\tilde{r})+{\rm log}^2|\tilde{r}(\xi-\zeta)|+{\rm log}^2|\tilde{r}(\xi+\zeta)|+{\rm log}(\tilde{r})+{\rm log}|\tilde{r}(\xi-\zeta)|+{\rm log}|\tilde{r}(\xi+\zeta)|+O(1)\\\label{e42}
&=: H_1(\tilde{r},\xi,\zeta)
\end{align}
for the case $m=3$, and 
\begin{align}\nonumber
&\left|A_{\varepsilon}(g,\tilde{r},\xi,\zeta)\right|\\\nonumber
&\leq \left[\left(\tilde{r}/c_2\right)^{-(3-m)}+O(1)\right]^2+\left[\Big|\frac{\tilde{r}}{2c_2}(\xi-\zeta)\Big|^{-(3-m)}+O(1)\right]^2+\left[\Big|\frac{\tilde{r}}{2c_2}(\xi+\zeta)\Big|^{-(3-m)}+O(1)\right]^2\\\nonumber
&\lesssim \tilde{r}^{-2(3-m)}+|\tilde{r}(\xi-\zeta)|^{-2(3-m)}+|\tilde{r}(\xi+\zeta)|^{-2(3-m)}+\tilde{r}^{-(3-m)}+|\tilde{r}(\xi-\zeta)|^{-(3-m)}\\
&\quad +|\tilde{r}(\xi+\zeta)|^{-(3-m)}+O(1)\nonumber\\\label{e43}
&=: H_2(\tilde{r},\xi,\zeta)
\end{align}
for the case $m\in (14/5, 3)$. According to facts 
\begin{eqnarray*}
&&\int_{0<\tilde{r}<c_2M}\left(\log^2(\tilde r)+\log(\tilde r)\right)d\tilde r<\infty,\\
&&\int_{\mathbb S^2}\int_{\mathbb S^2}\left(|\xi-\zeta|^{-2(3-m)}+|\xi-\zeta|^{-(3-m)}\right)dS(\xi)dS(\zeta)<\infty,
\end{eqnarray*} 
we get
\begin{eqnarray}\label{e44}
&&\int_{|g|<M}\int_{|\tilde{r}|<c_2M}\int_{\mathbb S^2}\int_{\mathbb S^2}\tilde{r}^2{\bf 1}_{[0,\infty)}(\tilde{r})H_j(\tilde{r},\xi,\zeta)\xi_1^{p_1}\xi_2^{p_2}\xi_3^{p_3}\zeta_1^{p_1}\zeta_2^{p_2}\zeta_3^{p_3}dS(\xi)dS(\zeta)d\tilde{r}dg\\\nonumber
&&<\infty
\end{eqnarray}
for $j=1,2$.
A direct application of Lemma \ref{lemmaadd5} gives that 
\begin{align}\nonumber
&\lim_{\varepsilon\to 0}A_{\varepsilon}(g,\tilde{r},\xi,\zeta)\\\nonumber
&=K_{\rho}\Big(\frac{g+(\tilde{r}\xi)/c_2}{2},\frac{g-(\tilde{r}\xi)/c_2}{2}\Big)K_{\rho}\Big(\frac{g+(\tilde{r}\zeta)/c_2}{2},\frac{g-(\tilde{r}\zeta)/c_2}{2}\Big)\\\nonumber
&\quad +K_{\rho}\Big(\frac{g+(\tilde{r}\xi)/c_2}{2},\frac{g+(\tilde{r}\zeta)/c_2}{2}\Big)K_{\rho}\Big(\frac{g-(\tilde{r}\xi)/c_2}{2},\frac{g-(\tilde{r}\zeta)/c_2}{2}\Big)\\\nonumber
&\quad +K_{\rho}\Big(\frac{g+(\tilde{r}\xi)/c_2}{2},\frac{g-(\tilde{r}\zeta)/c_2}{2}\Big)K_{\rho}\Big(\frac{g-(\tilde{r}\xi)/c_2}{2},\frac{g+(\tilde{r}\zeta)/c_2}{2}\Big)\\\label{e45}
&=:A(g,\tilde{r},\xi,\zeta)
 \end{align}
point-wisely. Hence, the dominated convergence theorem, together with (\ref{e42})--(\ref{e45}), leads to
\begin{align}\nonumber
&\lim_{\varepsilon\to 0}\int_{|g|<M}\int_{|\tilde{r}|<c_2M}\mathbb E\left|f_{\varepsilon}\left(g,\tilde{r}/c_2\right)\right|^2d\tilde{r}dg\\\nonumber
&=\int_{|g|<M}\int_{|\tilde{r}|<c_2M}\lim_{\varepsilon\to 0}\mathbb E\left|f_{\varepsilon}\left(g,\tilde{r}/c_2\right)\right|^2d\tilde{r}dg\\\nonumber
&=\frac{1}{64c_2^2}\int_{|g|<M}\int_{|\tilde{r}|<c_2M}\tilde{r}^2{\bf 1}_{[0,\infty)}(\tilde{r})\int_{\mathbb S^2}\int_{\mathbb S^2}A(g,\tilde{r},\xi,\zeta)\xi_1^{p_1}\xi_2^{p_2}\xi_3^{p_3}\zeta_1^{p_1}\zeta_2^{p_2}\zeta_3^{p_3}dS(\xi)dS(\zeta)d\tilde{r}dg\\\label{e46}
&<\infty
\end{align}
according to Lemma \ref{lemmaadd4}. Thus, we conclude that (\ref{e29}) holds for $\mathbb F_{\varepsilon}=\mathbb I_{\varepsilon}$, which leads to (\ref{e16}).

(ii) For the case $\mathbb F_{\varepsilon}=\mathbb J_{\varepsilon}$, it holds
\begin{align*}
\mathbb J_{\varepsilon}(\omega,\theta)&=\frac{1}{8}\omega^{-2}\int_{\mathbb R^3}\int_{\mathbb R^3}\rho_{\varepsilon}\Big(\frac{g+h}{2}\Big)\rho_{\varepsilon}\Big(\frac{g-h}{2}\Big)e^{{\rm i}c_1\omega  \theta\cdot g}\beta(r,\omega)r^{-3}\xi_1^{p_1}\xi_2^{p_2}\xi_3^{p_3}dhdg\\
&=\frac{1}{8}\omega^{-2}\int_{\mathbb R^3}\int_0^{\infty}\int_{\mathbb S^2}\rho_{\varepsilon}\Big(\frac{g+r\xi}{2}\Big)\rho_{\varepsilon}\Big(\frac{g-r\xi}{2}\Big)e^{{\rm i}c_1\omega  \theta\cdot g}\beta(r,\omega)r^{-1}\xi_1^{p_1}\xi_2^{p_2}\xi_3^{p_3}dS(\xi)dr dg\\
&=\omega^{-2}\int_{\mathbb R^3}e^{{\rm i}c_1\omega  \theta\cdot g}\int_{\mathbb R}\beta(r,\omega)f_{\varepsilon}(g,r)dr dg\\
&=\omega^{-2}\int_{\mathbb R^3}e^{{\rm i}c_1\omega  \theta\cdot g}Tf_{\varepsilon}(g,\omega)dg,
\end{align*}
where we use the fact $p_1+p_2+p_3-p_4=-3$ for $(p_1,p_2,p_3,p_4)\in S_{\mathbb J}$, and $f_{\varepsilon}$ and $Tf_{\varepsilon}$ are defined by
\begin{align*}
f_{\varepsilon}(g,r)&:=\frac{1}{8}r^{-1}{\bf 1}_{[0,\infty)}(r)\int_{\mathbb S^2}\rho_{\varepsilon}\Big(\frac{g+r\xi}{2}\Big)\rho_{\varepsilon}\Big(\frac{g-r\xi}{2}\Big)\xi_1^{p_1}\xi_2^{p_2}\xi_3^{p_3}dS(\xi),\\
Tf_{\varepsilon}(g,\omega)&:=\int_{\mathbb R}\beta(r,\omega)f_{\varepsilon}(g,r)dr.
\end{align*}
Thus, we obtain 
\begin{align}\nonumber
\int_{\mathbb S^2}\omega^{m-1}\left|\mathbb J_{\varepsilon}(\omega,\theta)\right|^2dS(\theta)&=\int_{\mathbb S^2}\omega^{m-5}\left|\int_{\mathbb R^3}e^{{\rm i}\omega c_1 \theta\cdot g}Tf_{\varepsilon}(g,\omega)dg\right|^2dS(\theta)\\\label{e50}
&\lesssim \omega^{m-7}\int_{\mathbb R^3}\left|Tf_{\varepsilon}(g,\omega)\right|^2dg,
\end{align}
where $Tf_{\varepsilon}$ is estimated as the one in (\ref{eq:Tf}). Since 
\begin{align*}
\beta(r,\omega)&= e^{{\rm i}\kappa_{\rm s}r}({\rm i}\kappa_{\rm s}r-1)-e^{{\rm i}\kappa_{\rm p}r}({\rm i}\kappa_{\rm p}r-1)\\
&= {\rm i}c_{\rm s}\omega r e^{{\rm i}c_{\rm s}\omega r}-{\rm i}c_{\rm p}\omega r e^{{\rm i}c_{\rm p}\omega r}+\left(e^{{\rm i}c_{\rm p}\omega r}-e^{{\rm i}c_{\rm s}\omega r}\right),
\end{align*}
we decompose $Tf_{\varepsilon}$ into three parts
\begin{eqnarray*}
Tf_{\varepsilon}(g,\omega)=T_1f_{\varepsilon}(g,\omega)-T_2f_{\varepsilon}(g,\omega)+T_3f_{\varepsilon}(g,\omega)
\end{eqnarray*}
where
\begin{align*}
T_1f_{\varepsilon}(g,\omega)&:={\rm i}c_{\rm s}\omega\int_{\mathbb R}r e^{{\rm i}c_{\rm s}\omega r}f_{\varepsilon}(g, r)d r,\\
T_2f_{\varepsilon}(g,\omega)&:={\rm i}c_{\rm p}\omega\int_{\mathbb R} r e^{{\rm i}c_{\rm p}\omega r}f_{\varepsilon}(g, r)d r,\\
T_3f_{\varepsilon}(g,\omega)&:=\int_{\mathbb R}\left(e^{{\rm i}c_{\rm p}\omega r}-e^{{\rm i}c_{\rm s}\omega r}\right)f_{\varepsilon}(g, r)d r.
\end{align*}
Hence,
\begin{align}\label{e56}
|Tf_{\varepsilon}(g,\omega)|^2\lesssim |T_1f_{\varepsilon}(g,\omega)|^2+|T_2f_{\varepsilon}(g,\omega)|^2+|T_3f_{\varepsilon}(g,\omega)|^2.
\end{align}
Substituting (\ref{e56}) into (\ref{e50}), taking the integration with respect to $\omega$ on the interval $[1,\infty)$, and taking the expectation on both sides of (\ref{e50}), we get
\begin{align}\label{e57}
&\mathbb E\int_1^{\infty}\int_{S^2}\omega^{m-1}\left|\mathbb J_{\varepsilon}(\omega,\theta)\right|^2dS(\theta)d\omega\notag\\
&\lesssim \mathbb E\int_1^{\infty}\omega^{m-7}\int_{\mathbb R^3}|T_1f_{\varepsilon}(g,\omega)|^2dgd\omega+\mathbb E\int_1^{\infty}\omega^{m-7}\int_{\mathbb R^3}|T_2f_{\varepsilon}(g,\omega)|^2dgd\omega\notag\\
&\quad +\mathbb E\int_1^{\infty}\omega^{m-7}\int_{\mathbb R^3}|T_3f_{\varepsilon}(g,\omega)|^2dgd\omega\notag\\
&=:I_{1, \varepsilon}+I_{2, \varepsilon}+I_{3, \varepsilon}.
\end{align}

Next, we estimate the above three terms separately. The first term satisfies 
\begin{align*}
I_{1, \varepsilon}&\lesssim\mathbb E\int_1^{\infty}\omega^{m-7}\int_{\mathbb R^3}\bigg|\int_{\mathbb R}c_{\rm s}\omega e^{{\rm i}c_{\rm s}\omega r}{\bf 1}_{[0,\infty)}( r)\int_{\mathbb S^2}\rho_{\varepsilon}\Big(\frac{g+ r\xi}{2}\Big)\rho_{\varepsilon}\Big(\frac{g- r\xi}{2}\Big)\xi_1^{p_1}\xi_2^{p_2}\xi_3^{p_3}dS(\xi)d r\bigg|^2dgd\omega\\
&=\mathbb E\int_1^{\infty}\omega^{m-5}\int_{\mathbb R^3}\bigg|\int_{\mathbb R} e^{{\rm i}\omega\tilde{r}}{\bf 1}_{[0,\infty)}(\tilde{r})\int_{\mathbb S^2}\rho_{\varepsilon}\Big(\frac{g+\frac{\tilde{r}\xi}{c_{\rm s}}}{2}\Big)\rho_{\varepsilon}\Big(\frac{g-\frac{\tilde{r}\xi}{c_{\rm s}}}{2}\Big)\xi_1^{p_1}\xi_2^{p_2}\xi_3^{p_3}dS(\xi)d\tilde{r}\bigg|^2dgd\omega\\
&\lesssim\mathbb E\int_{\mathbb R^3}\int_1^{\infty}\bigg|\int_{\mathbb R} e^{{\rm i}\omega\tilde{r}}{\bf 1}_{[0,\infty)}(\tilde{r})\int_{\mathbb S^2}\rho_{\varepsilon}\Big(\frac{g+\frac{\tilde{r}\xi}{c_{\rm s}}}{2}\Big)\rho_{\varepsilon}\Big(\frac{g-\frac{\tilde{r}\xi}{c_{\rm s}}}{2}\Big)\xi_1^{p_1}\xi_2^{p_2}\xi_3^{p_3}dS(\xi)d\tilde{r}\bigg|^2d\omega dg\\
&=\mathbb E\int_{\mathbb R^3}\int_{|\tilde{r}|<c_{\rm s}M}\bigg|\int_{\mathbb S^2}\rho_{\varepsilon}\Big(\frac{g+\frac{\tilde{r}\xi}{c_{\rm s}}}{2}\Big)\rho_{\varepsilon}\Big(\frac{g-\frac{\tilde{r}\xi}{c_{\rm s}}}{2}\Big)\xi_1^{p_1}\xi_2^{p_2}\xi_3^{p_3}dS(\xi)\bigg|^2d\tilde{r}dg\\
&=\int_{|g|<M}\int_{|\tilde{r}|<c_{\rm s}M}\int_{\mathbb S^2\times\mathbb S^2}\mathbb E\bigg[\rho_{\varepsilon}\Big(\frac{g+\frac{\tilde{r}\xi}{c_{\rm s}}}{2}\Big)\rho_{\varepsilon}\Big(\frac{g-\frac{\tilde{r}\xi}{c_{\rm s}}}{2}\Big)\rho_{\varepsilon}\Big(\frac{g+\frac{\tilde{r}\zeta}{c_{\rm s}}}{2}\Big)\rho_{\varepsilon}\Big(\frac{g-\frac{\tilde{r}\zeta}{c_{\rm s}}}{2}\Big)\bigg]\\
&\quad \times\xi_1^{p_1}\xi_2^{p_2}\xi_3^{p_3}\zeta_1^{p_1}\zeta_2^{p_2}\zeta_3^{p_3}dS(\xi)dS(\zeta)d\tilde{r}dg,
\end{align*}
where we use the fact $\omega^{m-5}\leq 1$ for $\omega\geq 1$ and $m\in (14/5, 3]$, and the Plancherel identity. By the same analysis as (\ref{e46}), we get 
\begin{eqnarray}\label{e59}
\lim_{\varepsilon\to 0}I_{1, \varepsilon}<\infty,
\end{eqnarray}
and similarly
\begin{eqnarray}\label{e60}
\lim_{\varepsilon\to 0}I_{2, \varepsilon}<\infty.
\end{eqnarray}
For the third term, we have from the mean-value theorem that 
\begin{align*}
e^{{\rm i}c_{\rm p}\omega r}-e^{{\rm i}c_{\rm s}\omega r}&=\cos(c_{\rm p}\omega r)-\cos(c_{\rm s}\omega r)+{\rm i}\left[\sin(c_{\rm p}\omega r)-\sin(c_{\rm s}\omega r)\right]\\
&=(c_{\rm p}-c_{\rm s})\omega r\left[-\sin(\eta_1)+{\rm i}\cos(\eta_2)\right],
\end{align*}
where $\eta_1$ and $\eta_2$ are values between $c_{\rm p}\omega r$ and $c_{\rm s}\omega r$. It follows from the Cauchy--Schwartz inequality that 
\begin{align*}
I_{3, \varepsilon}
&=\mathbb E\int_1^{\infty}\omega^{m-7}\int_{|g|<M}\bigg|\int_{| r|<M}(e^{{\rm i}c_{\rm p}\omega r}-e^{{\rm i}c_{\rm s}\omega r})f_{\varepsilon}(g, r)d r\bigg|^2dgd\omega\\
&=\mathbb E\int_1^{\infty}\omega^{m-7}\int_{|g|<M}\bigg|\int_{| r|<M}(c_{\rm p}-c_{\rm s})\omega r\left[-\sin(\eta_1)+{\rm i}\cos(\eta_2)\right]f_{\varepsilon}(g, r)d r\bigg|^2dgd\omega\\
&\lesssim \mathbb E\int_1^{\infty}\omega^{m-5}d\omega\int_{|g|<M}{\bf 1}_{[0,\infty)}( r)\bigg|\int_{\mathbb S^2}\rho_{\varepsilon}\Big(\frac{g+r\xi}{2}\Big)\rho_{\varepsilon}\Big(\frac{g-r\xi}{2}\Big)\xi_1^{p_1}\xi_2^{p_2}\xi_3^{p_3}dS(\xi)\bigg|^2dr dg\\
&\lesssim \frac{1}{4-m}\int_{|g|<M}\int_{|r|<M}{\bf 1}_{[0, \infty)}(r)\int_{\mathbb S^2\times\mathbb S^2}\mathbb E\Big[\rho_{\varepsilon}\Big(\frac{g+r\xi}{2}\Big)\rho_{\varepsilon}\Big(\frac{g-r\xi}{2}\Big)\\
&\quad \times\rho_{\varepsilon}\Big(\frac{g+r\zeta}{2}\Big)\rho_{\varepsilon}\Big(\frac{g-r\zeta}{2}\Big)\Big]\xi_1^{p_1}\xi_2^{p_2}\xi_3^{p_3}\zeta_1^{p_1}\zeta_2^{p_2}\zeta_3^{p_3}dS(\xi)dS(\zeta)dr dg.
\end{align*}
By the similar arguments as (\ref{e46}), it can be easily checked that 
\begin{eqnarray}\label{e63}
\lim_{\varepsilon\to 0}I_{3, \varepsilon}<\infty.
\end{eqnarray}
Then we conclude from (\ref{e57})--(\ref{e63}) that (\ref{e29}) holds for $\mathbb F_{\varepsilon}=\mathbb J_{\varepsilon}$, which leads to (\ref{e17}).
\end{proof}

\end{document}